\documentclass[preprint,12pt]{article}

\usepackage{amssymb}
\usepackage{mathrsfs}
\usepackage{amsfonts}
\usepackage{graphicx}
\usepackage{indentfirst}
\usepackage{colortbl,dcolumn}
\usepackage{color}
\usepackage{amsmath}
\usepackage{psfrag}
\usepackage{booktabs}
\usepackage{cite}
\usepackage{amsthm}
\usepackage{float}
\usepackage{verbatim}
\numberwithin{equation}{section}
\usepackage[top=1in, bottom=1in, left=0.8in, right=0.8in]{geometry}
\newcommand{\R}{\mathbb{R}}
\newcommand{\N}{\mathbb{N}}
\newcommand{\E}{\mathbb{E}}
\newcommand{\dd}{\text{d}}
\newtheorem{thm}{Theorem}[section]
\newtheorem{defn}[thm]{Definition}
\newtheorem{lem}[thm]{Lemma}
\newtheorem{prop}[thm]{Proposition}
\newtheorem{cor}[thm]{Corollary}
\newtheorem{rem}[thm]{Remark}

\newtheorem{assumption}[thm]{Assumption}

\begin{document}

\title{Order-One Convergence of the Backward Euler Method for  Random Periodic Solutions of Semilinear SDEs
\footnotemark[2] \footnotetext[2]{This work was supported by Natural Science Foundation of China (12071488, 12371417, 11971488) and
                Natural Science Foundation of Hunan Province (2020JJ2040). 
                \\
                E-mail addresses:
 y.j.guo@csu.edu.cn, x.j.wang7@csu.edu.cn, 
yue.wu@strath.ac.uk. 
                }}
\author{Yujia Guo$^{a}$, Xiaojie Wang$^{*a}$, and Yue Wu$^{b}$
	\\
\footnotesize $^a$ School of Mathematics and Statistics, HNP-LAMA, Central South University, Changsha, \\
\footnotesize Hunan, P. R. China \\
\footnotesize $^b$ Department of Mathematics and Statistics, University of Strathclyde, Glasgow G1 1XH, UK} 
\date{\today}

\maketitle

\begin{abstract}
In this paper, we revisit the backward Euler method for 
numerical approximations of random periodic 
solutions of semilinear SDEs with additive noise. 
Improved $L^p$-estimates of the random periodic solutions of the considered SDEs are obtained under a more relaxed condition compared to literature. The backward Euler scheme is proved to converge with an order one in the mean square sense, which also improves the existing order-half convergence. 
Numerical examples are presented to verify our theoretical analysis.

\end{abstract}

\section{Introduction}
Many phenomena in the real world {\color{black}exhibit} both periodic and random nature, {\color{black} for instance,} the daily temperature, energy consumption, and airline passenger volumes. 
{\color{black} To understand the long-time behaviour of these laws of random motion, where the underlying dynamics are usually modeled via stochastic differential equations (SDEs), it is crucial to study their random periodic solutions\cite{zhao2009random, feng2011pathwise}.}

{\color{black}Though the periodic solution has been a central concept
in the field of deterministic dynamical systems since Poincar\'e’s seminal work \cite{poincare1881memoire}, its stochastic counterpart has not been properly defined or studied until the last decade.}
Zhao and Zheng \cite{zhao2009random} {\color{black}first formulated} the definition of the pathwise random periodic solutions for $C^{1}$-cocycles {\color{black}of random dynamical systems and Feng, Zhao and Zhou \cite{feng2011pathwise} further developed the definition for semiflows.}
This pioneering study boosts a series of work, including the study of anticipating random solutions of SDEs with multiplicative linear noise \cite{feng2016anticipating},
the existence of random solutions generated by non-autonomous SPDEs with additive noise \cite{feng2012random}, periodic measures and ergodicity \cite{feng2020random}, etc. 
Note that almost all the sequel works are based on the definition of pathwise random periodic solutions for semiflows of random dynamical systems.

As the random periodic solution is not constructed explicitly, it is useful to study its numerical approximation. {\color{black}This is  an infinite time horizon problem. The relevant research on the numerics of random periodic solutions for SDEs has made rapid progress recently} \cite{mil1975approximate,burrage2004numerical,wang2020mean,wang2023mean,platen1999introduction,oksendal2003stochastic,ruemelin1982numerical}.
{\color{black} The first paper \cite{feng2017numerical} to approximate the random period trajectory considered an Euler-Maruyama method and a modified Milstein method for a dissipative system with a global Lipschitz condition.}
Wei and Chen \cite{wei2020numerical} later generalised the Euler-Maruyama method to the stochastic theta method {\color{black}by showing that the approximation converges to the exact one at an order $1/4$.}
Wu \cite{wu2022backward} studied the existence and uniqueness of the random periodic solution
for an additive SDE with a one-sided Lipschitz condition {\color{black}and gave the analysis of the order-half} convergence of its numerical approximation via the backward Euler method.

In this paper, we mainly consider the strong convergence rate of the backward Euler method for the random periodic solution of semilinear SDE with a one-sided Lipschitz condition. {\color{black}Our contribution is two-folded:
\begin{itemize}
    \item 
Improved $L^p$-estimates of the random periodic solutions of the considered SDEs 
    can be guaranteed under a more relaxed condition compared to \cite{wu2022backward};
    \item The order of convergence of the backward Euler method used to approximate random periodic solutions can be lifted from half in \cite{wu2022backward} to one for SDEs with an additive noise.
\end{itemize}}
The outline of the paper is as follows. 
In section \ref{sec:Random Periodic Solutions of SDEs}, we present some standard notation and assumptions that will be employed in our proofs,
and give the existence and uniqueness of the random periodic solution.
In Section \ref{sec:L^p-estimates}we consider the $L^{p}$-estimates of the random periodic solutions {\color{black}under a relaxed condition compared to \cite{wu2022backward}}. 
Section \ref{sec:strong-rate-BEM} is {\color{black} devoted to the error analysis of order-one convergence of the backward Euler method, where the supporting evidence is shown via}
numerical experiments in Section \ref{sec:numerical results}. 

\section{Random Periodic Solutions of SDEs}
\label{sec:Random Periodic Solutions of SDEs}
Let us recall the definition of the random periodic solution for stochastic semi-flows  given in \cite{feng2011pathwise}.
Let $X$ be a separable Banach space.
Denote by
$ {\color{black}(}\Omega,\mathcal{F},\mathbb{P},(\theta _{s})_{s\in \mathbb{R}}{\color{black})}$ 
a metric dynamical system and 
$\theta _{s}:\Omega \rightarrow \Omega$ 
is assumed to be a measurably invertible 
for all $s \in \mathbb{R}$.
Denote $\Delta:= \{(t,s)\in \mathbb{R}^{2},s\leq t \}$.  
Consider a stochastic  periodic semi-flow 
$u:\Delta \times \Omega \times X \rightarrow X $ of period $\tau$, 
which satisfies the following standard condition
\begin{equation}
    u(t,r,\omega) =
    u(t,s,\omega) 
    \circ
    u(s,r,\omega) \quad {\color{black}a.e.\ \omega \in \Omega,}
\end{equation}
for all $r \leq s \leq t, r, s, t\in \mathbb{R}$.

\begin{defn}\label{def:rps}
	{\color{black}Given $\tau>0$.} A random periodic solution of period $\tau$ of a semi-flow
	 $u:\Delta \times \Omega \times X \rightarrow X 
       $
	  is an $\mathcal{F}$-measurable map 
	  $Y: \mathbb{R} \times \Omega \rightarrow X $ 
	  such that
	\begin{equation}
		u(t-s,s,Y(s,\omega),\omega)=Y(t,\omega)\text{ and }
		Y(t+\tau,\omega)=
		Y(t,\theta_{\tau}\omega)
	\end{equation}
	for any $(t,s) \in \Delta$ and a.e. $\omega \in \Omega$.
\end{defn}
Throughout this paper the following notation
is frequently used. First, we have $ {\color{black}d} \in \mathbb{N} $.
Denote $[d]:=\{1,...,d\}$ and {\color{black}let the letter
$ C $ to denote a generic positive
constant independent of the time and temporal stepsize of the numerical scheme we will introduce later. }
Let {\color{black}$|\cdot|$,} $ \| \cdot \| $ and  
$ \left\langle \cdot , \cdot \right\rangle $
be {\color{black}the absolute value of a scalar,} the Euclidean norm and the inner product of vectors
in $ \R^d $, respectively.
By $A^{T}$ we denote the transpose of vector or matrix.
Given a matrix $A$, we use $\| A \|:=\sqrt{trace(A^{T}A)}$ to denote the trace norm of $A$.
On a canonical probability space 
$ ( \Omega, \mathcal{ F }, \mathbb{P} ) $, 
we use $ \E $ to {\color{black}denote the} expectation and
$ L^p(\Omega; \textcolor{black}{\R^{d }}) $
to denote the family of 
$ \textcolor{black}{\R^{d}} $-valued variables with
the norm defined by 
$ \| \xi \|_{L^p(\Omega;\textcolor{black}{\R^{d}})} 
=(\E [\|\xi \|^p])^{\frac{1}{p}} < \infty $.  {\color{black}
For a integrable random variable $X$ on $(\Omega,\mathcal{F},\mathbb{P})$ and for $\mathcal{G}$ a $\sigma$-algebra such that $\mathcal{G} \subset \mathcal{F}$, we use $\E[X|\mathcal{G}]$ to denote the conditional expectation.}
Let $W:\mathbb{R} \times \Omega \rightarrow
\mathbb{R}^{d} $ be a standard two-sided Wiener process on $(\Omega,\mathcal{F},\mathbb{P})$. 
The filtration is defined as
$\mathcal{F}^{t}_{s}:=
\sigma\{W_{u}-W_{v}:s\leq v\leq u\leq t\}$ and $\mathcal{F}^{t}
=\mathcal{F}^{t}_{-\infty}
=\bigvee_{s\leq t}\mathcal{F}^{t}_{s}$.
{\color{black}Denote the standard $\mathbb{P}$-preserving ergodic Wiener shift by 
$\theta:\mathbb{R} \times \Omega \rightarrow \Omega$,
which has the property $\theta_{t}(\omega)(s):=W_{t+s}-W_{t}$ for any $s,t \in \mathbb{R}$.}

{\color{black} Given a standard two-sided Wiener process $W$ on a metric dynamical system $ (\Omega,\mathcal{F},\mathbb{P},(\theta _{s})_{s\in \mathbb{R}})$, where $\theta$ is a  $\mathbb{P}$-preserving ergodic Wiener shift, }
we {\color{black}revisit the long time behaviour of the following stochastic differential equation with an additive noise}: 
\begin{equation}
	\label{eq:Problem_SDE}
	\left\{
	\begin{aligned}
		\dd X_{t}^{t_{0}} 
         & = {\color{black}\big(}-\Lambda X_{t}^{t_{0}}
         +f (t, \,X_{t}^{t_{0}}){\color{black}\big)}\dd t
         + g ( t ) \, \dd W(t),
   \quad t \in (t_{0},T],\\
		 X_{t_{0}}^{t_{0}} & = \xi,
	\end{aligned}\right.
\end{equation}
where  
{\color{black} $f:\mathbb{R} \times \mathbb{R}^{d}\rightarrow \mathbb{R}^{d}$,  $g:\mathbb{R} \rightarrow  \mathbb{R}$ and $\Lambda$ is $d \times d$ matrix, and reconsider its numerical treatment. Note that we use $X_{t_1}^{t_0}$ to emphasise a process $X$ evaluated at $t_1$ which starts from $t_0$.}
The random initial value $\xi$ is assumed to be $\mathcal{F}^{t_{0}}$-measurable. {\color{black} Further conditions on the initial value, the drift and diffusion coefficients are collected below.}
\begin{assumption}\label{ass}
Suppose the following conditions are satisfied.

(\romannumeral1)
$\Lambda$ is self-adjoint and positive definite operator and
there exists a non-decreasing sequence $(\lambda_{i})_{i\in [d] } \subset \mathbb{R}$ of positive real numbers and an orthonormal basis $(e_{i})_{i \in [d]}$, such that
\begin{equation}
\label{eq:Lambda}
\Lambda e_{i} =\lambda_{i} e_{i}, 
\end{equation}
for every $i \in [d]$.

(\romannumeral2)
The drift coefficient functions $f$ is continuous and $f(t,x)=f(t+\tau,x)$. 
There exists a constant $0< c_{f} < \lambda_{1}$ such that for any $x,y \in \mathbb{R}^ {d}$ and $t \in [0,\tau)$
\begin{equation}
\label{eq:c_f}
\begin{split}
    \langle 
    x - y , f(t,x) - f(t,y) 
    \rangle
     &\leq 
     c_{f} \| x - y \|^{2},\\    
      \langle 
       x , f(t,x) 
        \rangle
     &\leq 
     c_{f}(1+\|x\|^{2}).
\end{split}
\end{equation} 

(\romannumeral3)
The diffusion coefficient functions $g$ is continuous and $g(t)=g(t+\tau)$ and there exists a constant $c_{g}>0$ such that $\sup_{ t \in [0, \tau) } \|g(t) \|
       \leq c_{g}$ and
\begin{equation}
\begin{split}
    \|g(t_{1})-g(t_{2})\|
      \leq c_{g} |t_{2}-t_{1} |,
      \quad
      \forall t_{1},t_{2} \in [0, \tau).
\end{split}
\end{equation}

(\romannumeral4)
There exists a constant $C^{*}>0$ 
 such that 
$\E \big[\|\xi\|^{2} \big]\leq C^{*}$.
\end{assumption}

\begin{assumption}
\label{ass:f(t,x)_c_{f}}
There exists a constant $\tilde{c_{f}} $ such that 
\begin{equation}
    \begin{split}
        \Bigg \|
        f(t,x)-
        \frac{\langle f(t,x),x\rangle}{ \|x \|^{2}} x 
        \Bigg \| 
        \leq 
        \tilde{c_{f}}(1+ \|x \|^{2}), 
    \end{split}
\end{equation}
for $ x \in \mathbb{R}^{d},t \in[0,\tau) $.
\end{assumption}

 {\color{black}Following a similar argument as in \cite[Proposition 7.1]{2017Strong}, the SDE \eqref{eq:Problem_SDE} admits a unique global semiflow under Assumption \ref{ass} and Assumption \ref{ass:f(t,x)_c_{f}}. 
 Note that by the variation of constant formula, the solution of \eqref{eq:Problem_SDE} is written as
\begin{equation}
    \label{eq:variation of comstant formula}
    X^{t_{0}}_{t}(\xi)=
    e^{-\Lambda (t-t_{0})}\xi
    +\int_{t_{0}}^{t}
    e^{-\Lambda (t-s)} f(s,X^{t_{0}}_{s})\,\dd s
    +\int_{t_{0}}^{t}
    e^{-\Lambda(t-s)} g(s)\,\dd W_{s}.
\end{equation}
These assumptions can ensure the existence and uniqueness of the random periodic solution, which has been proposed in \cite[Theorem 7]{wu2022backward}.

\begin{thm}\label{thm:unique_random_periodic_solution}
 Let Assumptions \ref{ass} 
 and Assumptions \ref{ass:f(t,x)_c_{f}} be hold, then there exists a unique random periodic solution $X_{t}^{*}(\cdot)\in L^{2}(\Omega)$ {\color{black}with the form \begin{equation}
    \label{eq:pull-back}
    X^{*}_{t}=
    \int_{-\infty}^{t}
    e^{-\Lambda(t-s)}f(s,X^{*}_{s})\,\dd s
    +\int_{-\infty}^{t}
    e^{-\Lambda(t-s)}g(s)\,\dd W_{s}
\end{equation}} such that {\color{black} $X^{*}$ is a limit of the pull-back 
$X^{-k\tau}_{t}(\xi)$ of \eqref{eq:Problem_SDE} when $k \rightarrow \infty$, ie,
} 
	\begin{equation}
		\lim_{k \rightarrow \infty }
        \E\big[
        \| X_{t}^{-k\tau}(\xi) -X_{t}^{*} \|^{2}
        \big]
        =0.
	\end{equation}
\end{thm}

\section{Improved $L^p$-estimates of the random periodic solutions}
\label{sec:L^p-estimates}
{\color{black}In this section, our objective is to demonstrate the time-uniform moment boundedness of the random periodic solutions. 
While the work of Wu (2022) established the uniform boundedness for the $p$-th moment of the SDE solution as referenced in \cite{wu2022backward}, our contribution lies in surpassing the limitation imposed on the drift functions (see Remark \ref{rem:moment-bounds} for details).
To accomplish this objective, we begin by deriving a generalized lemma, building upon the principles outlined in \cite[Lemma 8.1]{it1964stationary}.}

\begin{lem} 
\label{lem:ito lemma}
Let $m:[a,\infty) \rightarrow [0,\infty)$, 
$\psi:[a,\infty) \rightarrow [0,\infty)$ be nonnegative continuous functions for $a \in \mathbb{R}$.
If {\color{black} there exists a positive constant $\delta$ such that}
\begin{equation}
\label{eq:m(t)-m(s)}
     m(t)-m(s)
     \leq 
     - \delta\int_{s}^{t} m(u) \, \dd u  
     +\int_{s}^{t} \psi(u) \, \dd u   
     \qquad     
\end{equation}
{\color{black}for any $a \leq s < t < \infty$, } then
\begin{equation}
    \label{eq:m(a)}
        m(t) 
        \leq
        m(a) 
         + \int_{a}^{t} 
           e^ {-\delta(t-u) }\psi(u) \, \dd u.
\end{equation}
\end{lem}
\begin{proof}[Proof of Lemma \ref{lem:ito lemma}]

Denote 
$m_{1}(t)
:=m(a) 
+\int_{a}^{t} e^{-\delta(t-u) }\psi(u)\, \dd u$. {\color{black} It is easy to derive that}
\begin{equation}
    \begin{split}
        m_{1}'(t)
        & =
        -\delta \int_{a}^{t} 
        e^{-\delta(t-u) }\psi(u)\, \dd u
        +\psi (t)\\
        & =
        -\delta(m_1(t)-m(a)) + \psi(t).
    \end{split}
\end{equation}
So we get
\begin{equation}\label{eqn:differencem1}
    m_{1}(t)-m_{1}(s)
    =
     -\delta \int_{s}^{t} (m_{1}(u)-m(a))\, \dd u 
     +\int_{s}^{t} \psi(u)\, \dd u.
\end{equation}
{\color{black} Now set } $m_{2}(t)=m(t)-m_{1}(t)$. 
{\color{black}
Due to  \eqref{eq:m(t)-m(s)}, \eqref{eqn:differencem1} and noting $m(a) \geq 0$,} we have
\begin{equation}
\label{eq:m_{2}}
      \begin{split}
      m_{2}(t)-m_{2}(s)
      & =m(t)-m(s)-(m_{1}(t)-m_{1}(s))\\
      & \leq
       -\delta \int_{s}^{t} (m(u)-m_{1}(u)+m(a))
       \, \dd u\\
      & \leq
       -\delta  \int_{s}^{t} (m(u)-m_{1}(u))\, \dd u\\
      & =
       -\delta \int_{s}^{t} m_{2}(u)\, \dd u.
      \end{split}
\end{equation}
To prove \eqref{eq:m(a)}, it {\color{black} suffices to show} that $m_{2}(t)\leq 0$ for any $t \in \mathbb{R}$.
If $m_{2}(t) > 0$ for some $t$, 
then $m_{2}(a)=0$ implies that there exists an interval $ [s_{1},t_{1}] \subset [a,\infty],\,m_{2}(t_{1}) >m_{2}(s_{1})$ and $ m_{2}>0 $ on $ [s_{1}, t_{1}]$, which contradict \eqref{eq:m_{2}}.
\end{proof}

{\color{black}We now give} the time-uniform moment boundedness of the exact solution {\color{black} of} {\color{black}SDE} \eqref{eq:Problem_SDE}.
\begin{thm} 
\label{thm:The p-th moment of the SDE solution}
Let Assumptions 
\ref{ass} hold. Let $X^{-k\tau}_{t}$ be the solution  of SDE \eqref{eq:Problem_SDE} with the initial value $X^{-k\tau}_{-k\tau}=\xi$ obeying
$\E \big[\|\xi\|^{2p} \big]< \infty$ for any $p\in [1,\infty)$. Then there exists a positive constant $C$ {\color{black} such that}
\begin{equation}
\label{eq:the_p_th_(X(t))}
\sup_{k\in \N}
\sup_{t\geq -k\tau}\E
\big[ 
{\| X^{-k\tau}_{t} \|}^{2p} 
\big]
\leq
C\big( 1+\E [\|\xi\|^{2p}] \big).
\end{equation}
\end{thm}

\begin{proof}[Proof of Theorem \ref{thm:The p-th moment of the SDE solution}]
Using the It\^o formula gives
\begin{equation}
	\begin{split}
		 \big(1+ {\| X^{-k\tau}_{t}\|}^{2}\big)^{p}
		=&
		 \big(1+ {\|\xi\|}^{2}\big)^{p}
		  +2p \int_{-k\tau}^{t}
           \big(1+{\| X^{-k\tau}_{s}\|}^{2}\big)^{p-1}
           \langle 
           X^{-k\tau}_{s},
           -\Lambda X^{-k\tau}_{s}
           \rangle
           \, \dd s \\
		&
		 +2p \int_{-k\tau}^{t}
          \big( 1+ {\| X^{-k\tau}_{s}\|}^{2} \big)^{p-1}
          \langle 
          X^{-k\tau}_{s},
          f(s,X^{-k\tau}_{s})
          \rangle
          \, \dd s \\		
		&
		 +2p\int_{-k\tau}^{t}
          \big( 1+{\| X^{-k\tau}_{s}\|}^{2} \big)^{p-1}
		 \langle 
		 X^{-k\tau}_{s},
		 g(s)       
		 \, \dd W_{s}
		 \rangle \\
		&
		 +p\int_{-k\tau}^{t}
          \big( 1+{\|X^{-k\tau}_{s}\|}^{2} \big)^{p-1}
          \| g(s) \|^{2}
          \, \dd s\\
		&
		 +2p(p-1)\int_{-k\tau}^{t}
          \big( 1+{\|X^{-k\tau}_{s}\|}^{2} \big)^{p-2}
		 \| ( X^{-k\tau}_{s}) ^{T}g(s)\|^{2} 
		 \, \dd s.
	\end{split}
\end{equation}
{\color{black} This directly leads to }
\begin{equation}
	\begin{split}
		 \big( 1+{\| X^{-k\tau}_{t}\|}^{2} \big)^{p}
		\leq 
        &
		 \big(1+{\|\xi\|}^{2} \big)^{p}
		 +2p \int_{-k\tau}^{t}
          \big( 1+{\| X^{-k\tau}_{s}\|}^{2} \big)^{p-1}
		 \langle 
		 X^{-k\tau}_{s},
		 -\Lambda X^{-k\tau}_{s}
		 \rangle
		 \, \dd s \\
		&
		 +2p\int_{-k\tau}^{t}
          \big( 1+ {\| X^{-k\tau}_{s}\|}^{2} \big)^{p-1}
		 \langle 
		 X^{-k\tau}_{s},
		 f(s,X^{-k\tau}_{s})
		 \rangle
		 \, \dd s \\		
		&
		 +2p\int_{-k\tau}^{t}
          \big(1+ {\| X^{-k\tau}_{s}\|}^{2} \big)^{p-1}
		 \langle 
		 X^{-k\tau}_{s},
		 g(s)       
		 \, \dd W_{s}
		 \rangle \\
		&
		 +p(2p-1)\int_{-k\tau}^{t}
          \big(1+ {\| X^{-k\tau}_{s}\|}^{2} \big)^{p-1}
		 \| g(s) \|^{2} 
		 \, \dd s .\\
	\end{split}
\end{equation}
Combining  this with \eqref{eq:Lambda} 
and \eqref{eq:c_f} 
we can get
\begin{equation}
    \begin{split}
         \big( 1+\|X^{-k\tau}_{t}\|^{2} \big)^{p}
        \leq 
        &
          \big( 1+\|\xi\|^{2} \big)^{p}
          +2p\int_{-k\tau}^{t}
          \big( 1+\|X^{-k\tau}_{s}\|^{2} \big)^{p-1}
          \big(
         -\lambda_{1} 
         (1+\|X^{-k\tau}_{s}\|^{2}) 
         \big)
         \, \dd s \\
        &
         +2pc_{f}\int_{-k\tau}^{t}
         \big( 1+\|X^{-k\tau}_{s}\|^{2} \big)^{p}
         \, \dd s \\
        &
         +2p\int_{-k\tau}^{t}
         \big( 1+ {\| X^{-k\tau}_{s}\|}^{2} \big)^{p-1}
         \langle 
         X^{-k\tau}_{s},
         g(s)       
         \, \dd W_{s}
         \rangle \\
        &
         +p 
         \big(
         (2p-1)c_{g}^{2} + 2\lambda_{1}
         \big)
         \int_{-k\tau}^{t}
         \big( 1+ {\| X^{-k\tau}_{s}\|}^{2} \big)^{p-1}
         \, \dd s\\
		\leq 
        &
		 \big( 1+ {\|\xi\|}^{2} \big)^{p}
		 -2p(\lambda_{1}-c_{f})
		 \int_{-k\tau}^{t}
          \big( 1+ {\| X^{-k\tau}_{s}\|}^{2} \big)^{p}
		 \, \dd s \\
		&
		 +2p\int_{-k\tau}^{t}
          \big( 1+ {\| X^{- k\tau}_{s}\|}^{2}\big)^{p-1}
		 \langle 
		 X^{-k\tau}_{s},
		 g(s)       
		 \, \dd W_{s}
		 \rangle \\
		&
		 +p\big((2p-1)c_{g}^{2} + 2\lambda_{1}\big)
		 \int_{-k\tau}^{t}
          \big( 1+ \|{ X^{-k\tau}_{s}\|}^{2}\big)^{p-1}
		 \, \dd s,
	\end{split}
\end{equation}
{\color{black}where
\begin{equation}
    \begin{split}
        & p
        \big(
        (2p-1)c_{g}^{2} + 2\lambda_{1}
        \big)
        \big( 
        1+\|{ X^{-k\tau}_{s}\|}^{2}
        \big)^{p-1}\\
        & \quad=
        p(\lambda_{1}-c_{f})
        \big( 
        1+\|{ X^{-k\tau}_{s}\|}^{2}
        \big)^{p-1}
        \times
        \tfrac{(2p-1)c_{g}^{2} + 2\lambda_{1}}{\lambda_{1}-c_{f}}
        \\
        & \quad \leq
        p(\lambda_{1}-c_{f})
        \times \tfrac{p-1}{p}
        \big( 
        1+\|{ X^{-k\tau}_{s}\|}^{2}
        \big)^{p-1}
        +
        p(\lambda_{1}-c_{f})
        \times \tfrac{1}{p}
        \tfrac{((2p-1)c_{g}^{2} + 2\lambda_{1})^{p}}{(\lambda_{1}-c_{f})^{p}}\\
        & \quad =   
        (p-1)(\lambda_{1}-c_f)
        \big( 
        1+\|{ X^{-k\tau}_{s}\|}^{2} 
        \big)^{p}
        +
        \tfrac
         {((2p-1)c_{g}^{2} + 2\lambda_{1})^{p}}{(\lambda_{1}-c_f)^{p-1}}
    \end{split}
\end{equation}
by the Young inequality
$a^{p-1}b \leq \frac{p-1}{p}a^{p}+\frac{1}{p}b^{p}$.
} Therefore
\begin{equation}\label{eq:no E}
	\begin{split}
		 \big( 1+ {\|X^{-k\tau}_{t}\|}^{2}\big)^{p}
		\leq 
        &
		 \big( 1+ {\|\xi\|}^{2} \big)^{p}
		 -(p+1)(\lambda_{1}-c_{f})
          \int_{-k\tau}^{t}
          \big( 1+\|{ X^{-k\tau}_{s}\|}^{2} \big)^{p}
		 \, \dd s \\
		&
		 +\int_{-k\tau}^{t} 
          \tfrac{((2p-1)c_{g}^{2}+2\lambda_{1}))^{p}}{(\lambda_{1}-c_{f})^{p-1}}
		 \, \dd s\\
		&
		 +2p \int_{-k\tau}^{t}
          \big( 1+ {\| X^{-k\tau}_{s}\|}^{2} \big)^{p-1}
		 \langle 
		 X^{-k\tau}_{s},
		 g(s)       
		 \, \dd W_{s}
		 \rangle .\\		
	\end{split}
\end{equation}
For every interval $n\geq 1$, define the stopping time
\begin{equation}
\tau_{n}= \inf \{s\in [-k\tau,t]:
{\color{black}\|X_{s}^{-k\tau}\|}\geq n\}.
\end{equation}
Clearly, $\tau_{n} \uparrow t$ a.s.. Moreover, 
it follows from (\ref{eq:no E}) and the property of the It\^o integral that
\begin{equation}
    \begin{split}
        \E[(1+ 
        {\| X^{-k\tau}_{t \wedge \tau_{n}}\|}^{2})^{p}]
        \leq
        &
         \E[ (1+ {\|\xi\|}^{2})^{p} ]
         -(p+1) (\lambda_{1}-c_{f})
         \E \Big
         [\int_{-k\tau}^{t \wedge \tau_{n}}
         \big( 1+\| { X^{-k\tau}_{s}\|}^{2} \big)^{p} 
         \, \dd s \Big]\\
        & 
         +\E\Big[
         \int_{-k\tau}^{\color{black}t \wedge \tau_{n}} 
         \tfrac{((2p-1)c_{g}^{2}+2\lambda_{1})^{p}}{(\lambda_{1}-c_{f})^{p-1}}
         \, \dd s \Big].
    \end{split}
\end{equation}
Letting $n \rightarrow \infty$ and by the {\color{black}Fatou's} lemma, we have
\begin{equation}\label{eqn:fatou}
	\begin{split}
		 \E[(1+ {\| X^{-k\tau}_{t}\|}^{2})^{p}]
		\leq
        &
		 \E[ (1+ {\|\xi\|}^{2})^{p} ]
		 -(p+1) (\lambda_{1}-c_{f})
		 \int_{-k\tau}^{t}
		 \E \Big[
          (1+\| { X^{-k\tau}_{s}\|}^{2})^{p}
          \Big]
		 \, \dd s \\
		& 
		 +\int_{-k\tau}^{t} 
          \tfrac{((2p-1)c_{g}^{2}+2\lambda_{1})^{p}}
		 {(\lambda_{1}-c_{f})^{p-1}}
		 \, \dd s .
	\end{split}
\end{equation}
{\color{black}Now applying Lemma \ref{lem:ito lemma} to \eqref{eqn:fatou} with $\delta=(p+1)(\lambda_{1}-c_{f} )>0$ and $\psi=\tfrac{((2p-1)c_{g}^{2}
	+2\lambda_{1})^{p}}{(\lambda_{1}-c_{f})^{p-1}}$ gives}
 \begin{equation}
     \begin{split}
          \E[(1+ {\| X^{-k\tau}_{t}\|}^{2})^{p}]
         & \leq
           \E[ (1+ {\|\xi\|}^{2})^{p} ]
           +\int_{-k\tau}^{t}
           e^{-
        (p+1)(\lambda_{1}-c_{f})(t-u)}
           \tfrac{((2p-1)c_{g}^{2}+2\lambda_{1})^{p}}
           {(\lambda_{1}-c_{f})^{p-1}}
          \,\dd u\\
         & \leq
           \E[ (1+ {\|\xi\|}^{2})^{p} ]
           +\tfrac{((2p-1)c_{g}^{2}+2\lambda_{1})^{p}}
           {(p+1)(\lambda_{1}-c_{f})^{p}}
           (1-e^{-(p+1)(\lambda_{1}-c_{f})(t+k\tau)})\\
         & \leq
           C\big( 1+\E [\|\xi\|^{2p}] \big),
     \end{split}
 \end{equation}
which completes the proof.
\end{proof}
{\color{black}
\begin{rem}
\label{rem:moment-bounds}
The above time-uniform moment boundedness was obtained under more relaxed conditions, compared to
\cite[Proposition 14]{wu2022backward},
where the author additionally required $(c_f + \tfrac{(p-1)c_g}{2} ) ( 2 + p + 2^{p+1}) < p \lambda_1$ for some
positive number $p \geq 4 \gamma - 2$ (cf. \cite[Assumption 13]{wu2022backward}).
Instead, here we just require $0< c_{f} < \lambda_{1}$, which significantly relaxes  
the aforementioned condition.
\end{rem}
}

\section{The order-one convergence of the backward Euler method}
\label{sec:strong-rate-BEM}
{\color{black} This section is devoted to the error analysis of the strong convergence rate of the backward Euler approximation to SDE \eqref{eq:Problem_SDE}, where we lift the order of convergence to one from half as obtained in \cite{wu2022backward}.  }
Take an  equidistant partition 
$\mathcal{T}^{h}:=\{jh,j \in \mathbb{Z}\}$, 
such that $h\in(0,1)$. 
Note that $\mathcal{T}^{h}$ stretch along the real line because we are dealing with an infinite time horizon problem.
The backward Euler method  applied to SDE \eqref{eq:Problem_SDE} takes the following form: 
\begin{equation}
\label{eq:backward_Euler_method}
{\tilde{X}}
^{-k\tau}_{-k\tau+jh}
=
 {\tilde{X}}
 ^{-k\tau}_{-k\tau+(j-1)h}
 -\Lambda
 h{\tilde{X}}^{-k\tau}_{-k\tau+jh}
+ hf
 \big( jh, 
 {\tilde{X}}^{-k\tau}_{-k\tau+jh}
 \big) 
+g((j-1)h)\Delta W_{-k\tau+(j-1)h}
\end{equation}
for all $ j \in \mathbb{N}$, 
where 
$\Delta W_{-k\tau+(j-1)h}:=
W_{-k\tau+jh}-W_{-k\tau+(j-1)h}$,
and the initial value 
$\tilde{X}^{-k\tau}_{-k\tau}=\xi$.
Because of the periodicity of $f$ and $g$, we have that $f(-k\tau+jh,{\tilde{X}}^{-k\tau}_{-k\tau+jh}) =f(jh,{\tilde{X}}^{-k\tau}_{-k\tau+jh})$, 
$g(-k\tau+jh) =g(jh)$.

{\color{black}Before proceeding to the assumptions and the main proof of the error analysis, we present some existing and essential results regarding \eqref{eq:backward_Euler_method} from literature. The uniform bounds for the second moment of the numerical approximation have been established in \cite[Lemma 17]{wu2022backward} as follows.}
\begin{prop}\label{prop:second moment of NS}
Let Assumptions \ref{ass} be satisfied.
Then there exists {\color{black}a constant} $\tilde{C}>0$ such that
\begin{equation}
\sup_{k,j\in \mathbb{N}}
\E[\| \tilde{X}^{-k\tau}_{-k\tau+jh} \|^2]
\leq
\tilde{C},
\end{equation}
where $\{\tilde{X}^{-k\tau}_{-k\tau+jh}\}_{k,j \in \mathbb{N}}$ is given by \eqref{eq:backward_Euler_method}.
\end{prop}

{\color{black} The existence and uniqueness of random periodic solutions associated to the backward Euler method has been guaranteed in \cite{wu2022backward}. }
\begin{thm}
\label{thm:BEM random period solution}
Let Assumptions \ref{ass}  be satisfied. For $h \in (0,1)$, the time domain is divided as {\color{black}$\mathcal{T}^{h}$}.
The backward Euler method \eqref{eq:backward_Euler_method} admits a random period solution $\tilde{X}^{*}\in L^{2}(\Omega)$ such that
\begin{equation}
\lim_{k \rightarrow \infty }
 \E\big[
 \|\tilde{X}_{-k\tau+jh}^{-k\tau}(\xi)
 -\tilde{X}^{*}\|^{2}
 \big]
 =0.
\end{equation}
\end{thm}

In order to establish a strong convergence rate {\color{black} as high as one} of the backward Euler method, we need the following conditions {\color{black} on the drift as well as the initial condition besides Assumption \ref{ass}.}

\begin{assumption}\label{ass:f^{'}}
Assume the drift coefficient functions $f:\mathbb{R} \times \mathbb{R}^{d}\rightarrow \mathbb{R}^{d}$ are continuously differentiable, and 
there exists a constant $\gamma \in [1,\infty)$ such that
\begin{align}
\label{eq:the_esti_f'(t,x)-f'(t,y)}
\big\|
\big(
\tfrac{\partial f}{\partial x}
(t,x)
-\tfrac{\partial f}{\partial x}(t,\bar{x}{\color{black})}
\big)
y
\big\|
&\leq
C(1+\| x \| + \| \bar{x} \| )^{\gamma -2} 
\|
x - \bar{x}
\|
\|y \|,
\quad \forall x,\bar{x},y\in \mathbb{R}^{d}, \\
\label{eq:the_esti_f(t,x)-f(s,x)}
\|f(t,x) - f(s,x)\|
&\leq
C(1+\| x \| ^{\gamma}) |t-s| ,
\quad \forall x \in \mathbb{R}^{d},
s,t \in [0,\tau),
\end{align}
{\color{black} and }
$\| \xi \|_{L^{\max\{ 4\gamma, 8\gamma-8 \} }(\Omega;\R^{d})}
<\infty $.
\end{assumption}
Note that the condition \eqref{eq:the_esti_f'(t,x)-f'(t,y)}
implies
\begin{equation}
\label{eq:the_esti_f'(t,x)}
\| 
\tfrac{\partial f}{\partial x}
(t,x) y
\|
\leq
C(1+\|x \|)
^{\gamma-1}
\|y\|,
\quad \forall x,y \in  \mathbb{R}^{d}, \\
\end{equation}
which in turn implies
\begin{align}
\label{eq:the_esti_f(t,x)-f(t,y)}
\| f(t,x)-f(t,\bar{x}) \|
&\leq
C(1+\| x \| + \| \bar{x} \| )
^{\gamma -1}
\|
x - \bar{x}
\|,
\quad \forall x,\bar{x} \in  \mathbb{R}^{d},\\
\label{eq:the_esti_f(t,x)}
\| f(t,x) \|
&\leq
C(1+\|x \|)
^{\gamma},
\quad \forall x\in  \mathbb{R}^{d}.
\end{align}

{\color{black}We now present two lemmas with estimates} which
play important role in  proving the order of convergence of the backward Euler method in Theorem \ref{thm:error analysis}.
{\color{black} The first one is the} 
H\"older continuity of the exact solution of
\eqref{eq:Problem_SDE} with respect to the norm in 
$L^p(\Omega;\mathbb{R}^d)$ {\color{black} and the second one is about the regularity of $f$}.
\begin{lem}
\label{lem:the_esti_(X(t1)-X(t2)}
Let
Assumptions \ref{ass}  and \ref{ass:f^{'}}   
be hold.
Then there exists a positive constant $C$ 
which depends on $\gamma,d,\Lambda,f,g$ only, 
such that
\begin{equation}
\label{eq:the_esti_(X(t1)-X(t2)}
\|
X_{t_1}^{-k\tau}-
X_{t_2}^{-k\tau} 
\|
_{L^p(\Omega;\mathbb{R}^d)}
\leq
C
\Big(
1+\sup_{k\in \mathbb{N}} \sup_{t\geq -k\tau}
\| 
X_{t}^{-k\tau} 
\|
^\gamma_{L^{p\gamma}(\Omega;\mathbb{R}^d)}
\Big)
|t_{2}-t_{1}|
+
C |t_{2}-t_{1}|^{\frac{1}{2}},
\end{equation}
for all $t_1,t_2 \geq -k\tau$ and
{\color{black} $p \in [2,\infty)$},
where $X_{t}^{-k\tau}$ denotes the exact solution to {\color{black}the SDE}
\eqref{eq:Problem_SDE}.
\end{lem}

\begin{proof}[Proof of Lemma \ref{lem:the_esti_(X(t1)-X(t2)}]
Without loss of generality we set $t_1 \leq t_2$ and get
\begin{equation}\label{eqn:difftrue}
\begin{split}
\|
X_{t_1}^{-k\tau}-X_{t_2}^{-k\tau} 
\|
_{L^p(\Omega;\mathbb{R}^d)}   
& =
\Big\|
\int_{t_1}^{t_2}
\big(
-\Lambda X_{r}^{k\tau}
+f(r,X_{t}^{-k\tau}
\big) 
\,\dd r
+
\int_{t_1}^{t_2} g(r)\,\dd {\color{black}W_r}
\Big\|
_{L^p(\Omega;\mathbb{R}^d)} \\
& \leq
\Big\|
\int_{t_1}^{t_2}
\big(
\Lambda X_{r}^{-k\tau}+f(r,X_{t}^{-k\tau})
\big)
\,\dd r
\Big\|
_{L^p(\Omega;\mathbb{R}^d)} \\
& \quad +
\Big\|
\int_{t_1}^{t_2} g(r)\,\dd {\color{black}W_r}
\Big\|
_{L^p(\Omega;\mathbb{R}^d)}.
\end{split}   
\end{equation}
For the first term, 
{\color{black}we can get the following estimate} using the H\"older inequality, \eqref{eq:the_p_th_(X(t))} and 
\eqref{eq:the_esti_f(t,x)} 
\begin{equation}
\begin{split}
\Big\|
& \int_{t_1}^{t_2}
\big( 
-\Lambda X_{r}^{-k\tau} + f(r,X_{t}^{-k\tau})
\big)
\,\dd r
\Big\|_{L^p(\Omega;\mathbb{R}^d)} 
\\
& \quad \leq
\int_{t_1}^{t_2}
\|
\Lambda X_{r}^{-k\tau}
\|
_{L^p(\Omega;\mathbb{R}^d)}\, \dd r
+
\int_{t_1}^{t_2}
\|
f(r,X_{t}^{-k\tau})
\|
_{L^p(\Omega;\mathbb{R}^d)}\, \dd r \\
& \quad \leq
C
\Big(
1+\sup_{k\in \mathbb{N}} \sup_{t\geq -k\tau}
\| X_{t}^{-k\tau} \|^\gamma_{L^{p\gamma}(\Omega;\mathbb{R}^d)}
\Big)
|t_2 - t_1 |. 
\end{split}
\end{equation}
{\color{black}Applying the Burkholder-Davis-Gundy inequality to the last term of \eqref{eqn:difftrue} gives} 
\begin{equation}
\Big\|
\int_{t_1}^{t_2} g(r)\,\dd {\color{black}W_r}
\Big\|_{L^p(\Omega;\mathbb{R}^d)} 
\leq
C \Big (
\int_{t_1}^{t_2}
\| g(r) \|^{2}_{L^p(\Omega;\mathbb{R}^d)}
\, \dd r
\Big )^{\frac{1}{2}}
\leq
C |t_2 - t_1 |^{\frac{1}{2}}.
\end{equation}
This completes proof.
\end{proof}

\begin{lem}
\label{lem:the_esti_f(s,x(s))-f(t,x(t)}
Let Assumptions \ref{ass}  and \ref{ass:f^{'}}  be
satisfied. 
{\color{black}Consider the exact solution $X^{-k\tau}_t$ of SDE \eqref{eq:Problem_SDE} over $[-k\tau, T]$ for arbitrary $k\in \mathbb{N}$ and $T\geq -k\tau$.}
Then there exists a positive constant $C$ 
which depends on $\gamma,d,\Lambda,f,g$ only, 
such that for all $t_1,t_2 \in [-k\tau,T]$
and $s \in [t_1,t_2 ]$, 
it holds that
\begin{equation}
\label{eq:lem:the_esti_f(s,x(s))-f(t,x(t)}
\begin{split}
\big\|
f(s,X_{s}^{-k\tau})
-f(t_2,X_{t_2}^{-k\tau}) 
\big\|_{L^2(\Omega;\mathbb{R}^d)}
& \leq 
C
\Big(
1+\sup_{k\in \mathbb{N}} \sup_{t\geq -k\tau}
\| X_{t}^{-k\tau} \|^{\gamma-1}_{L^{4\gamma-2}(\Omega;\mathbb{R}^d)}
\Big)
|t_2-t_1|^{\frac{1}{2}} \\
& \quad +
C
\Big(
1+\sup_{k\in \mathbb{N}} \sup_{t\geq -k\tau}
\| X_{t}^{-k\tau} \|^{2\gamma-1}_{L^{4\gamma-2}(\Omega;\mathbb{R}^d)}
\Big)
|t_2-t_1|.
\end{split}
\end{equation}
\end{lem}

\begin{proof}[Proof of Lemma \ref{lem:the_esti_f(s,x(s))-f(t,x(t)}]
For all $s,t_2 \in [-k\tau,T]$,
it follows from 
\eqref{eq:the_esti_f(t,x)-f(t,y)} and
\eqref{eq:the_esti_f(t,x)}
that
\begin{equation}
\label{eq:the_esti_f(s,x(s))-f(t,x(t)}
\begin{split}
\big\|
f(s,X_{s}^{-k\tau})-
f(t_2,X_{t_2}^{-k\tau}) 
\big\|
& =
\big\|
f(s,X_{s}^{-k\tau})
-f(s,X_{t_2}^{-k\tau}) 
+f(s,X_{t_2}^{-k\tau})
-f(t_2,X_{t_2}^{-k\tau}) 
\big\| \\
& \leq
\big\|
f(s,X_{s}^{-k\tau})
-f(s,X_{t_2}^{-k\tau})
\big\|
+
\big\|
f(s,X_{t_2}^{-k\tau})
-f(t_2,X_{t_2}^{-k\tau}) 
\big\| \\
& \leq
C(1+\| X_{s }^{-k\tau}\| + \| X_{t_2}^{-k\tau}\| )^{\gamma -1}
\| X_{t_2}^{-k\tau} - X_{s}^{-k\tau} \| \\
& \quad +
C (1+\| X_{t_2}^{-k\tau}\|^{\gamma} )
|t_2 -s |.
\end{split}
\end{equation}
{\color{black}Taking the expectation on both sides and}
using the H\"older inequality 
{\color{black} $\|f^{\gamma-1}g\|_{L^2(\Omega;\mathbb{R}^d)} \leq
\|f\|^{\gamma-1}_{L^{2\rho_1(\gamma-1)}(\Omega;\mathbb{R}^d)}\\
+
\|g\|_{L^{2\rho_2}(\Omega;\mathbb{R}^d)}
\, (\frac{1}{2\rho_1}+\frac{1}{2\rho_2}=1)$ 
with exponents 
$\rho_1 =:\frac{2\gamma-1}{\gamma-1}$ and
$\rho_2 =:\frac{2\gamma-1}{\gamma}$ yield that} 
\begin{equation}
\begin{split}
& \big\|
f(s,X_{s}^{-k\tau})
-f(t_2,X_{t_2}^{-k\tau}) 
\big\|
_{L^2(\Omega;\mathbb{R}^d)} \\
& \quad \leq
C
\big\|
(
1
+\| X_{s}^{-k\tau}\| 
+ \| X_{t_2}^{-k\tau}\| 
)
^{\gamma -1}
\| X_{t_2 }^{-k\tau}-X_{s}^{-k\tau} \|
\big\| _{L^2(\Omega;\mathbb{R})} \\
& \qquad +
C 
\big\|
(1+\| X_{t_2}^{-k\tau}\|^{\gamma })
|t_2 -s |
\big\| _{L^2(\Omega;{\color{black}\mathbb{R})}} \\
& \quad  \leq 
C
\Big(
1+\sup_{k\in \mathbb{N}} \sup_{t\geq -k\tau}
\| X_{t}^{-k\tau}\|
^{\gamma-1}_{L^{2\rho_1(\gamma-1)}(\Omega;{\mathbb{R}^d})} 
\Big) 
\| 
X_{t_2 }^{-k\tau}-X_{s}^{-k\tau} 
\|
 _{L^{2\rho_2}(\Omega;{\color{black}\mathbb{R}^d})}\\
& \qquad +
 C
\Big(
 1+\sup_{k\in \mathbb{N}} \sup_{t\geq -k\tau}
\| X_{t}^{-k\tau}\| 
^{\gamma}_{L^{2\gamma}(\Omega;\mathbb{R}^d)}
\Big)
|t_2 -t_1 |.\\
\end{split}
\end{equation}
Moreover, through Lemma 
\ref{lem:the_esti_(X(t1)-X(t2)}  
with $p=2\rho_2$ {\color{black}we have that}
\begin{equation}
\begin{split}
\| X_{s}^{-k\tau}-X_{t_2}^{-k\tau} \|_{L^{2\rho_2}(\Omega;\mathbb{R}^d)} 
& \leq
C
\Big(
1+\sup_{k\in \mathbb{N}} \sup_{t\geq -k\tau}
\| X_{t}^{-k\tau}\| ^{\gamma}_{L^{2\gamma\rho_{2}}(\Omega;\mathbb{R}^d)}
\Big)
|t_2-s|
+C|t_2-s|^\frac{1}{2} \\
& \leq
C\Big(
1+\sup_{k\in \mathbb{N}} \sup_{t\geq -k\tau}
\| X_{t}^{-k\tau}\|^{\gamma}_{L^{4\gamma-2}(\Omega;\mathbb{R}^d)}
\Big)|t_2-t_1|
+C|t_2-t_1|^\frac{1}{2}. \\
\end{split}
\end{equation}
{\color{black}Note that $2\rho_1(\gamma-1)=4\gamma-2$. }Altogether, it follows that {\color{black}for $s \in [t_1,t_2]$}
\begin{equation}
\begin{split}
&\big\|
f(s,X_{s}^{-k\tau})
-f(t_2,X_{t_2}^{-k\tau}) 
\big\|_{L^2(\Omega;\mathbb{R}^d)}\\
& \leq 
C
\Big(
1+\sup_{k\in \mathbb{N}} \sup_{t\geq -k\tau}
\| X_{t}^{-k\tau} \|^{\gamma-1}_{L^{4\gamma-2}(\Omega;\mathbb{R}^d)}
\Big)
|t_2-t_1|^{\frac{1}{2}} \\
& \quad +
C
\Big(
1+\sup_{k\in \mathbb{N}} \sup_{t\geq -k\tau}
\big(\| X_{t}^{-k\tau} \|^{2\gamma-1}_{L^{4\gamma-2}(\Omega;\mathbb{R}^d)}
\Big)
|t_2-t_1|.
\end{split}
\end{equation}
This thus finishes
the proof of the lemma.
\end{proof}

We are now ready to give the main result of this section that reveals the order-one convergence of the backward Euler scheme to the SDE \eqref{eq:Problem_SDE} in the long run.

\begin{thm}\label{thm:error analysis}
Under Assumptions
\ref{ass}  and 
\ref{ass:f^{'}}.
For $h \in (0,1)$, the time domain is divided as {\color{black}$\mathcal{T}^{h}$}.
If  $X_{-k\tau+jh}^{-k\tau}$ and 
$\tilde{X}_{-k\tau+jh}^{-k\tau}$ 
are the exact and the numerical solutions given by \eqref{eq:Problem_SDE} and \eqref{eq:backward_Euler_method}, respectively, 
then
there exists a constant $C$ that depends on the $\gamma,\Lambda,f,g$ and $d$ such that 
\begin{equation}
\sup_{k,j}
\| 
X_{-k\tau+jh}^{-k\tau}
- \tilde{X}_{-k\tau+jh}^{-k\tau}
\|_{L^{2}(\Omega;\mathbb{R}^d)}
\leq Ch.
\end{equation}
\end{thm}

\begin{proof}[Proof of Theorem \ref{thm:error analysis}]
First note that
\begin{align}    
\label{eq:X_j-X_s}
\begin{split}
X_{-k\tau+jh}^{-k\tau} & =
X_{-k\tau+(j-1)h}^{-k\tau}
-\int_{-k\tau+(j-1)h}^{-k\tau+jh}\Lambda X_{s}^{-k\tau} \, \dd s\\
&\quad
+\int_{-k\tau+(j-1)h}^{-k\tau+jh}f(s,X_{s}^{-k\tau})
\, \dd s
+\int_{-k\tau+(j-1)h}^{-k\tau+jh}g(s) 
\, \dd W_{s}\\
& =
X_{-k\tau+(j-1)h}^{-k\tau}
-\Lambda hX_{-k\tau+jh}^{-k\tau}
+hf(jh,X_{-k\tau+jh}^{-k\tau})\\
&\quad+g((j-1)h)\Delta W_{-k\tau+(j-1)h}
+\mathcal{R}_{j},
\end{split}
\end{align}
where
\begin{equation}
\label{eq:R_j}
\begin{split}
\mathcal{R}_{j}
&:=-\int_{-k\tau+(j-1)h}^{-k\tau+jh}
\Lambda(X_{s}^{-k\tau}-X_{-k\tau+jh}^{-k\tau}) 
\, \dd s\\
& \quad+
\int_{-k\tau+(j-1)h}^{-k\tau+jh}
f(s,X_{s}^{-k\tau})-f(jh,X_{-k\tau+jh}^{-k\tau})
\, \dd s \\
& \quad 
+\int_{-k\tau+(j-1)h}^{-k\tau+jh}
g(s)-g((j-1)h) 
\, \dd W_{s}.\\
\end{split}
\end{equation}
Subtracting \eqref{eq:backward_Euler_method} from this yields
\begin{equation} \label{eq:error}
\begin{split}
X_{-k\tau+jh}^{-k\tau}- \tilde{X}_{-k\tau+jh}^{-k\tau}
&=
X_{-k\tau+(j-1)h}^{-k\tau}- \tilde{X}_{-k\tau+(j-1)h}^{-k\tau}-\Lambda h(X_{-k\tau+jh}^{-k\tau}- \tilde{X}_{-k\tau+jh}^{-k\tau})\\
& \quad
+h
\big(
f(jh,X_{-k\tau+jh}^{-k\tau})-f(jh,\tilde{X}_{-k\tau+jh}^{-k\tau})
\big)\\
& \quad
+g((j-1)h)\Delta W_{-k\tau+(j-1)h}-g((j-1)h)\Delta W_{-k\tau+(j-1)h}+\mathcal{R}_{j}.\\
\end{split}
\end{equation}
For brevity, we denote
\begin{equation}
\label{eq:D_j,Delta f_{j}}
\begin{split}
D_j
:=X_{-k\tau+jh}^{-k\tau}- \tilde{X}_{-k\tau+jh}^{-k\tau},
\quad
\Delta f_{j}
:=f(jh,X_{-k\tau+jh}^{-k\tau})-f(jh, \tilde{X}_{-k\tau+jh}^{-k\tau}).
\end{split}
\end{equation}
We emphasize that $D_j$ and $\Delta f_{j}$ are $\mathcal{F}_{jh-k\tau}$-measurable.\\
Using \eqref{eq:D_j,Delta f_{j}}, \eqref{eq:error} can be now rewritten as
\begin{equation}
D_j=D_{j-1}-\Lambda hD_j+h\Delta f_{j}+\mathcal{R}_{j}.
\end{equation}
{\color{black} This leads to}
\begin{equation}
\|
D_j+\Lambda hD_j-h\Delta f_{j}
\| ^{2}
=
\|
D_{j-1}+\mathcal{R}_{j}
\| ^{2}.
\end{equation}
Using \eqref{eq:Lambda} {\color{black}and} \eqref{eq:c_f} gives
\begin{equation}
\begin{split}
\|
&D_j+\Lambda hD_j-h\Delta f_{j}
\| ^{2} \\
&=\| D_j \| ^{2}+2h \langle D_{j},\Lambda D_j \rangle 
-2h \langle D_j,\Delta f_{j} \rangle+h^{2}\| \Lambda D_j -\Delta f_{j} \|^{2}\\
&\geq
\| D_j \| ^{2}+2\lambda_{1}h \| D_j \|^{2}-2c_{f}h \|D_j \|^{2}\\
&=
\big(1+2(\lambda_{1}-c_{f})h\big)\| D_j \| ^{2}.
\end{split}
\end{equation}
Meanwhile,
\begin{equation}
\|
D_{j-1}+\mathcal{R}_{j}
\| ^{2}
=\| D_{j-1}\|^{2}+2\langle D_{j-1},\mathcal{R}_{j} \rangle
+\| \mathcal{R}_{j} \| ^{2}.
\end{equation}
As a result,
\begin{equation}
\big(1+2(\lambda_{1}-c_{f})h\big)\| D_j\| ^{2}
\leq
\| D_{j-1} \| ^{2}+2\langle D_{j-1},\mathcal{R}_{j} \rangle
+\| \mathcal{R}_{j}\| ^{2}.
\end{equation} 
Denoting $v:=\lambda_{1}-c_{f}$ and taking expectation yield,
\begin{equation}
(1+2vh)\E[\| D_j\|^{2}]
\leq
\E[\| D_{j-1} \|^{2}]+
2\E[\langle D_{j-1},\mathcal{R}_{j} \rangle]
+\E[\| \mathcal{R}_{j}\|^{2}].
\end{equation}
Recalling ${\color{black}D_{j-1}}$ is {\color{black}$\mathcal{F}_{(j-1)h-k\tau}$}-measurable, we deduce
\begin{equation}
\E[\langle D_{j-1},\mathcal{R}_{j} \rangle]
=
\E 
\big[
\E [\langle D_{j-1},\mathcal{R}_{j} \rangle 
| \mathcal{F}_{(j-1)h-k\tau}] 
\big]
=
\E 
\big[{\color{black}\langle} D_{j-1},
 \E[\mathcal{R}_{j}| \mathcal{F}_{(j-1)h-k\tau}]\rangle
\big].
\end{equation}
Further, noting $v>0$ by  \eqref{eq:c_f}, and using the Cauchy-Schwartz inequality $2ab\leq vha^{2}+\frac{1}{vh}b^{2}$ {\color{black}for arbitrary positive $h$},
we obtain
\begin{equation}
\begin{split}
(1+2vh)\E[\| D_j\|^{2}]
&\leq
\E[\| D_{j-1}\|^{2}]+2\E \big[\langle \sqrt{vh} D_{j-1},\frac{1}{\sqrt{vh}}\E[\mathcal{R}_{j} | \mathcal{F}_{(j-1)h-k\tau}] \rangle \big ]
+\E[\| \mathcal{R}_{j} \| ^{2}]\\
&\leq
(1+vh)\E[\|D_{j-1}\|^{2}]+\E[\| \mathcal{R}_{j}\|^{2}]
+\dfrac{1}{vh}\E\big[\| \E[\mathcal{R}_{j} | \mathcal{F}_{{(j-1)h-k\tau}}] \|^{2}\big].
\end{split}
\end{equation}
Hence,
\begin{equation}\label{eq:D_j}
\begin{split}
\E[\|D_j \|^{2}]
&\leq
\tfrac{1+vh}{1+2vh}\E[\| D_{j-1}\|^{2}]
+\tfrac{1}{1+2vh}\E[\| \mathcal{R}_{j}\| ^{2}]+
\tfrac{1}{(1+2vh)vh}\E\big[\|\E[\mathcal{R}_{j} | \mathcal{F}_{(j-1)h-k\tau}]\|^{2}\big]\\
&=
(1-\tfrac{v}{1+2vh}h)\E[\| D_{j-1}\|^{2}]
+\tfrac{1}{1+2vh}\E[\| \mathcal{R}_{j}\| ^{2}]
+\tfrac{1}{(1+2vh)vh}\E\big[\| \E[\mathcal{R}_{j} | \mathcal{F}_{{(j-1)h-k\tau}}]\|^{2}\big].
\end{split}
\end{equation}
Therefore, we only need to estimate two error terms $\E[\| \mathcal{R}_{j}\|^{2}]$ and $\E\big[\|\E[\mathcal{R}_{j}|\mathcal{F}_{{(j-1)h-k\tau}}] \|^{2}\big]$.
Recalling the definition of $\mathcal{R}_{j}$ given by \eqref{eq:R_j}
and using an triangle inequality yield
\begin{equation}
\label{eq:|R_j|^2|}
\begin{split}
\| \mathcal{R}_{j} \|_{L^{2}(\Omega;\mathbb{R}^d)}
&\leq
\bigg \|
\int_{-k\tau+(j-1)h}^{-k\tau+jh}
\Lambda(X_{s}^{-k\tau}-X_{-k\tau+jh}^{-k\tau}){\color{black}\, \dd s}
\bigg \|
_{L^{2}(\Omega;\mathbb{R}^d)}
\\
& \quad +
\bigg\|
\int_{-k\tau+(j-1)h}^{-k\tau+jh}
f(s,X_{s}^{-k\tau})
-f(jh,X_{-k\tau+jh}^{-k\tau}) 
\, \dd s
\bigg\|
_{L^{2}(\Omega;\mathbb{R}^d)}\\
& \quad+
\bigg\|
\int_{-k\tau+(j-1)h}^{-k\tau+jh} 
g(s)-g((j-1)h) 
\, \dd W_{s} 
 \bigg\|
 _{L^{2}(\Omega;\mathbb{R}^d)}\\
&=:
\mathbb{I}_1
+\mathbb{I}_2
+\mathbb{I}_3.
\end{split}
\end{equation}
For the term
$\mathbb{I}_1$,  
following the H\"older inequality and \eqref{eq:the_esti_(X(t1)-X(t2)} shows
\begin{equation}
\begin{split}
\label{eq:I_1}
\mathbb{I}_1
& \leq
\int_{-k\tau+(j-1)h}^{-k\tau+jh}
\|
\Lambda(X_{s}^{-k\tau}-X_{-k\tau+jh}^{-k\tau}) 
\|_{L^{2}(\Omega;\mathbb{R}^d)}
\, \dd s\\
& \leq
Ch^{\frac{3}{2}}
\Big(
1+\sup_{k\in \mathbb{N}} \sup_{t\geq -k\tau}
\| X_{t}^{-k\tau} \|^\gamma_{L^{2\gamma}(\Omega;\mathbb{R}^d)}
\Big).
\end{split}
\end{equation}
With the help of 
the H\"older inequality
and \eqref{eq:the_esti_f(s,x(s))-f(t,x(t)},
one can obtain
\begin{equation}
\label{eq:I_2}
\begin{split}
\mathbb{I}_2
& \leq
\int_{-k\tau+(j-1)h}^{-k\tau+jh}
\bigg\|
f(s,X_{s}^{-k\tau})
-f(jh,X_{-k\tau+jh}^{-k\tau})
\bigg\|
_{L^{2}(\Omega;\mathbb{R}^d)} 
\, \dd s\\
& \leq
Ch^{\frac{3}{2}}
\Big(
1+\sup_{k\in \mathbb{N}} \sup_{t\geq -k\tau}
\| X_{t}^{-k\tau} \| 
^{2\gamma-1}_{L^{4\gamma-2}
(\Omega;\mathbb{R}^d)}
\Big).
\end{split}
\end{equation}
In view of the It\^o isomery and by Assumption \ref{ass} (iii)
\begin{equation}
\label{eq:I_3}
\begin{split}
\mathbb{I}_3
=
\Big( 
\int_{-k\tau+(j-1)h}^{-k\tau+jh} 
\| g(s)-g((j-1)h) \| ^{2} \, \dd s
\Big)
^{\frac{1}{2}}
\leq
Ch^{\frac{3}{2}}.
\end{split}
\end{equation}
Putting all the above estimates together we derive from \eqref{eq:|R_j|^2|} that 
\begin{equation}
\| 
\mathcal{R}_{j} 
\|
_{L^{2}(\Omega;\mathbb{R}^d)}
\leq
Ch^{\frac{3}{2}}
\Big(
1+\sup_{k\in \mathbb{N}} \sup_{t\geq -k\tau}
\| X_{t}^{-k\tau} \|
^{2\gamma-1}_{L^{4\gamma-2}(\Omega;\mathbb{R}^d)}
\Big).\\
\end{equation}
Note that
$
\E[\int_{-k\tau+(j-1)h}^{-k\tau+jh}
g(s)-g((j-1)h) 
\, \dd W_{s}  | \mathcal{F}_{(j-1)h-k\tau}]=0
$. 
Thus
\begin{equation}
\label{eq:R_j|F}
\begin{split}
& \|
\E
[\mathcal{R}_{j} | \mathcal{F}_{(j-1)h-k\tau}]
\|
_{L^{2}(\Omega ;\mathbb{R}^{d})} \\
& \leq
\Big\|
\E
\Big[
\int_{-k\tau+(j-1)h}^{-k\tau+jh}
-\Lambda 
(X_{s}^{-k\tau}-X_{-k\tau+jh}^{-k\tau}) 
\, \dd s
\Big| \mathcal{F}_{(j-1)h-k\tau}
\Big]
\Big\|
_{L^{2}(\Omega ;\mathbb{R}^{d})}\\
& \quad +
\Big\|
\E
\Big[
\int_{-k\tau+(j-1)h}^{-k\tau+jh}
f(s,X_{s}^{-k\tau})
-f(jh,X_{-k\tau+jh}^{-k\tau}) 
\, \dd s
\Big| \mathcal{F}_{(j-1)h-k\tau}
\Big]
\Big\| 
_{L^{2}(\Omega ;\mathbb{R}^{d})}\\
& =:
\mathbb{I}_4
+\mathbb{I}_5.
\end{split}
\end{equation}
In order to estimate $\mathbb{I}_4$, 
we first show that
\begin{equation}
\begin{split}
\label{eq:g(r)|F}
& \E
\Big[
\int_{-k\tau+(j-1)h}^{-k\tau+jh}
\int_{s}^{-k\tau+jh}
g(r) 
\, \dd {\color{black}W_r} \, \dd {\color{black}s}
\Big| \mathcal{F}_{(j-1)h-k\tau}
\Big] \\
& \quad =
\int_{-k\tau+(j-1)h}^{-k\tau+jh}
\E
\Big[
\int_{s}^{-k\tau+jh}
g(r) 
\, \dd {\color{black}W_r}
\Big| \mathcal{F}_{(j-1)h-k\tau}
\Big]
\, \dd {\color{black}s} \\
& \quad =
\int_{-k\tau+(j-1)h}^{-k\tau+jh}
\E
\Big[
\E
\Big[
\int_{s}^{-k\tau+jh}
g(r) 
\, \dd {\color{black}W_r}
\Big| \mathcal{F}_{s}
\Big]
\Big| \mathcal{F}_{(j-1)h-k\tau}
\Big]
\, \dd {\color{black}s}\\
& \quad =0.
\end{split}    
\end{equation}
As a result,
we derive from \eqref{eq:X_j-X_s} that 
\begin{equation}
\begin{split}
& \E
\Big[
\int_{-k\tau+(j-1)h}^{-k\tau+jh}
-\Lambda 
(X_{s}^{-k\tau}-X_{-k\tau+jh}^{-k\tau}) 
\, \dd s
\Big| \mathcal{F}_{(j-1)h-k\tau}
\Big] \\
& \quad=
\E
\Big [
\int_{-k\tau+(j-1)h}^{-k\tau+jh}
\int_{s}^{-k\tau+jh}
(
-\Lambda^2 X_{r}^{-k\tau}
+ \Lambda f(r,X_{r}^{-k\tau})
)
 \, \dd r \, \dd s 
\Big ] .\\
\end{split}
\end{equation}
Based on the Jensen inequality and 
the H\"older inequality, 
according to 
\eqref{eq:the_p_th_(X(t))} and
\eqref{eq:the_esti_f(t,x)}, one can infer
\begin{equation}
\label{eq:I4}
\begin{split}
\mathbb{I}_4
& =
\Big\|
\E
\Big [
\int_{-k\tau+(j-1)h}^{-k\tau+jh}
\int_{s}^{-k\tau+jh}
\big(
-\Lambda^2 X_{r}^{-k\tau}
+ \Lambda f(r,X_{r}^{-k\tau}) 
\big)
\, \dd r \, \dd s
\Big| \mathcal{F}_{(j-1)h-k\tau}
\Big ]
\Big\|
_{L^{2}(\Omega ;\mathbb{R}^{d})}\\
& \leq
\Big\|
\int_{-k\tau+(j-1)h}^{-k\tau+jh}
\int_{s}^{-k\tau+jh}
\big(
-\Lambda^2 X_{r}^{-k\tau}
+ \Lambda f(r,X_{r}^{-k\tau}) 
\big)
\, \dd r \, \dd s
\Big\|
_{L^{2}(\Omega ;\mathbb{R}^{d})}\\
& \leq
\int_{-k\tau+(j-1)h}^{-k\tau+jh}
\int_{s}^{-k\tau+jh}
\|
-\Lambda^2 X_{r}^{-k\tau}
+ 
\Lambda f(r,X_{r}^{-k\tau}) 
\|
_{L^{2}(\Omega ;\mathbb{R}^{d})}
\, \dd r \, \dd s \\
& \leq
\int_{-k\tau+(j-1)h}^{-k\tau+jh}
\int_{s}^{-k\tau+jh}
\|
\Lambda^2
X_{r}^{-k\tau}
\|
_{L^{2}(\Omega ;\mathbb{R}^{d})}
\, \dd r \, \dd s \\
& \quad +
\int_{-k\tau+(j-1)h}^{-k\tau+jh}
\int_{s}^{-k\tau+jh}
\|
\Lambda
f(r,X_{r}^{-k\tau}) 
\|
_{L^{2}(\Omega ;\mathbb{R}^{d})}
\, \dd r \, \dd s \\
& \leq
Ch^{2}
\Big(
1+
\big(\| X_{t}^{-k\tau} \|^{\gamma}_{L^{2\gamma}(\Omega;\mathbb{R}^d)}
\Big).
\end{split}
\end{equation}
With regard to $\mathbb{I}_5$, 
{\color{black} we first 
rewrite it as follows} 
\begin{equation}
\begin{split}
\mathbb{I}_5
& =
\Big\|
\E
\Big[
\int_{-k\tau+(j-1)h}^{-k\tau+jh}
f(s,X_{s}^{-k\tau})
-f(j,X_{-k\tau+jh}^{-k\tau}) 
\, \dd s
\Big| \mathcal{F}_{(j-1)h-k\tau}
\Big]
\Big\| 
_{L^{2}(\Omega ;\mathbb{R}^{d})}\\
& =
\Big\|
\E
\Big[
\int_{-k\tau+(j-1)h}^{-k\tau+jh}
f(s,X_{s}^{-k\tau})
-f(s,X_{-k\tau+jh}^{-k\tau})
\, \dd s 
\Big| \mathcal{F}_{(j-1)h-k\tau}
\Big]
\\
& \qquad +
\E
\Big[
\int_{-k\tau+(j-1)h}^{-k\tau+jh}
f(s,X_{-k\tau+jh}^{-k\tau})
-f(jh,X_{-k\tau+jh}^{-k\tau})
\, \dd s 
\Big| \mathcal{F}_{(j-1)h-k\tau}
\Big]
\Big\|
_{L^{2}(\Omega ;\mathbb{R}^{d})}.
\\
\end{split}
\end{equation}
{\color{black}Because of the existence of the first derivative of $f$ with respect to the spatial variable,}
for $s\in [-k\tau+(j-1)h,-k\tau+jh]$,
\begin{equation}\label{eq:f(t,y)-f(t,x)}
\begin{split}
& f(s,X_{-k\tau+jh}^{-k\tau})
 -f(s,X_{s}^{-k\tau})\\
& \quad =
\int_{0}^{1}
\tfrac{\partial f}{\partial x}
\big(s,
X_{s}^{-k\tau}
+
\zeta 
(X_{-k\tau+jh}^{-k\tau}-X_{s}^{-k\tau})
\big)
(X_{-k\tau+jh}^{-k\tau}
-X_{s}^{-k\tau})
\, \dd \zeta \\
& \quad =
\int_{0}^{1}
\tfrac{\partial f}{\partial x}
\big(s,
X_{s}^{-k\tau}
+
\zeta 
(X_{-k\tau+jh}^{-k\tau}-X_{s}^{-k\tau})
\big)
\int_{s}^{-k\tau+jh}
\big(
-\Lambda X_{r}^{-k\tau}
+f(r,X_{r}^{-k\tau}) 
\big)
\, \dd r \, \dd \zeta \\
& \qquad +
\int_{0}^{1}
\tfrac{\partial f}{\partial x}
\big(s,
X_{s}^{-k\tau}
+
\zeta (X_{-k\tau+jh}^{-k\tau}-X_{s}^{-k\tau})
\big)
\int_{s}^{-k\tau+jh}
g(r) 
\,\dd W_{r} \, \dd \zeta \\
& \quad =
\begin{matrix}
\underbrace
{
\int_{0}^{1}
\tfrac{\partial f}{\partial x}
\big(
s, X_{s}^{-k\tau}
+
\zeta 
(X_{-k\tau+jh}^{-k\tau}-X_{s}^{-k\tau})
\big)
\int_{s}^{-k\tau+jh}
\big(
-\Lambda X_{r}^{-k\tau}
+f(r,X_{r}^{-k\tau}) 
\big)
\, \dd r \, \dd \zeta  
}_{=: \Phi_{1}}
\end{matrix}
\\
& \qquad +
\int_{0}^{1}
\tfrac{\partial f}{\partial x}
\big(s,X_{s}^{-k\tau}
\big)
\int_{s}^{-k\tau+jh}
g(r)
\, \dd W_{r} \, \dd \zeta \\
&  \qquad +
\begin{matrix}
\underbrace
{
\int_{0}^{1}
\Big(
\tfrac{\partial f}{\partial x}
\big(
s,X_{s}^{-k\tau}
+ \zeta  (X_{-k\tau+jh}^{-k\tau}-X_{s}^{-k\tau})
\big)
- \tfrac{\partial f}{\partial x}
\big(
s,X_{s}^{-k\tau}
\big)
\Big)
\int_{s}^{-k\tau+jh}
g(r)
\, \dd W_{r} \, \dd \zeta 
}_{= : \Phi_{2}}. \\
\end{matrix}
\end{split}
\end{equation}
{\color{black}Following a similar argument} as \eqref{eq:g(r)|F}, we have that
\begin{equation}
\E
\Big[
\int_{-k\tau+(j-1)h}^{-k\tau+jh}
\int_{0}^{1}
\tfrac{\partial f}{\partial x}
\big(s,
X_{s}^{-k\tau}
\big)
\int_{s}^{-k\tau+jh}
g(r)
\,\dd W_{r} \,\dd \zeta 
\, \dd s
\Big| \mathcal{F}_{(j-1)h-k\tau}
\Big ]
=0.
\end{equation}
Using the Jensen inequality 
and the H\"older inequality {\color{black}gives}
\begin{equation}
\label{eq:I_5}
\begin{split}
\mathbb{I}_5
& \leq
\Big\|
\E
\Big[
\int_{-k\tau+(j-1)h}^{-k\tau+jh}
\Phi_1
\Big| \mathcal{F}_{(j-1)h-k\tau}
\Big]
\Big\|
_{L^{2}(\Omega ;\mathbb{R}^{d})}
 +
\Big\|
\E
\Big[
\int_{-k\tau+(j-1)h}^{-k\tau+jh}
\Phi_2
\Big| \mathcal{F}_{(j-1)h-k\tau}
\Big]
\Big\|
_{L^{2}(\Omega ;\mathbb{R}^{d})}
\\
& \quad +
\Big\|
\E
\Big[
\int_{-k\tau+(j-1)h}^{-k\tau+jh}
f(s,X_{-k\tau+jh}^{-k\tau})
-f(jh,X_{-k\tau+jh}^{-k\tau})
\, \dd s 
\Big| \mathcal{F}_{(j-1)h-k\tau}
\Big]
\Big\|
_{L^{2}(\Omega ;\mathbb{R}^{d})}
\\
& \leq
\Big\|
\int_{-k\tau+(j-1)h}^{-k\tau+jh}
\Phi_1
\, \dd s 
\Big\|
_{L^{2}(\Omega ;\mathbb{R}^{d})}
+
\Big\|
\int_{-k\tau+(j-1)h}^{-k\tau+jh}
\Phi_2
\, \dd s
\Big\|
_{L^{2}(\Omega ;\mathbb{R}^{d})} \\
& \quad +
\Big\|
\int_{-k\tau+(j-1)h}^{-k\tau+jh}
f(s,X_{-k\tau+jh}^{-k\tau})
-f(jh,X_{-k\tau+jh}^{-k\tau})
\, \dd s
\Big\|_{L^{2}(\Omega ;\mathbb{R}^{d})}\\
& \leq
\int_{-k\tau+(j-1)h}^{-k\tau+jh}
\|\Phi_1\|
_{L^{2}(\Omega ;\mathbb{R}^{d})}
\, \dd s 
+
\int_{-k\tau+(j-1)h}^{-k\tau+jh}
\| \Phi_2 \|
_{L^{2}(\Omega ;\mathbb{R}^{d})}
\, \dd s \\
& \quad +
\int_{-k\tau+(j-1)h}^{-k\tau+jh}
\|
f(s,X_{-k\tau+jh}^{-k\tau})
-f(jh,X_{-k\tau+jh}^{-k\tau})
\|
_{L^{2}(\Omega ;\mathbb{R}^{d})}
\, \dd s.
\end{split}
\end{equation}
In the following, we cope with the last three terms separately.
According to  the H\"older inequality and
\eqref{eq:the_p_th_(X(t))},
\eqref{eq:the_esti_f'(t,x)} {\color{black} and}
\eqref{eq:the_esti_f(t,x)}, 
we can get
\begin{equation}
\label{eq:Phi_1}
    \begin{split}
        \|
        \Phi_1
        \|
        _{L^{2}(\Omega ;\mathbb{R}^{d})} 
        & \leq
        \int_0^1
        \int_{s}^{-k\tau+jh}
        \big\|
        \tfrac{\partial f}{\partial x}
        \big(s,X_{s}^{-k\tau}
        +\zeta (
        X_{-k\tau+jh}^{-k\tau}-X_{s}^{-k\tau}) {\color{black}\big)}\\
        & \qquad 
        {\color{black}\times}
        \big(
        -\Lambda X_{r}^{-k\tau}
        +f(r,X_{r}^{-k\tau}) 
        \big)
        \big\|
        _{L^{2}(\Omega ;\mathbb{R}^{d})}
        \, \dd r \, \dd \zeta \\
        & \leq
        \int_0^1
        \int_{s}^{-k\tau+jh}
        \big\|
        C
        \big(
        1+ 
        \|
        \zeta X_{-k\tau+jh}^{-k\tau}
        +(\zeta -1) X_{s}^{-k\tau}
        \|
        \
        \big)
        ^{\gamma-1}
        \\
        & \qquad        
        {\color{black}\times}
        \big(
        1+
        \|
        X_{r}^{-k\tau}
        \|
        \big)^{\gamma}
        \big\|
        _{L^{2}(\Omega ;\mathbb{R})}
        \, \dd r \, \dd \zeta \\
        &\leq
        Ch 
        \Big(
        1+\sup_{k\in \mathbb{N}} \sup_{t\geq -k\tau}
        \| X_{t}^{-k\tau} \| 
        ^{2\gamma-1}_{L^{4\gamma-2}(\Omega;\mathbb{R}^d)}
        \Big).
    \end{split}
\end{equation}
 For the second term 
 $\|\Phi_2\|_{L^{2}(\Omega ;\mathbb{R}^{d})}$,
%
{\color{black}we can get the following estimate through} the H\"older inequality, Burkholder-Davis-Gundy inequality,
\eqref{eq:the_esti_f'(t,x)-f'(t,y)} and
\eqref{eq:the_esti_(X(t1)-X(t2)},
\begin{equation}
\label{eq:phi}
\begin{split}
\|
\Phi_2
\|
_{L^{2}(\Omega ;\mathbb{R}^{d})}
& \leq
\int_0^1
\Big\|
{\color{black}\Big(}
\tfrac{\partial f}{\partial x}
\big(
s,X_{s}^{-k\tau}
+ \zeta  (X_{-k\tau+jh}^{-k\tau}-X_{s}^{-k\tau})
\big)
- \tfrac{\partial f}{\partial x}
\big(
s,X_{s}^{-k\tau}
\big)
{\color{black}\Big)}\\
& \qquad
{\color{black}\times}\int_{s}^{-k\tau+jh}
g(r)
\, \dd W_{r} 
\Big \|
_{L^{2}(\Omega ;\mathbb{R}^{d})} 
\, \dd \zeta
\\
& \leq
{\color{black}C}\int_{0}^{1}
\Big\|
\big(1+
 \| 
 \zeta X_{-k\tau+jh}
 ^{-k\tau}+(\zeta -1) X_{s}^{-k\tau}
 \|
 +\| X_{s}^{-k\tau} \|
\big)^{\gamma -2} \\
& \qquad 
{\color{black}\times}\| 
\zeta (X_{-k\tau+jh} 
^{-k\tau}-X_{s}^{-k\tau}) 
\| 
\Big\|
_{L^{4}(\Omega ;\mathbb{R})}
\Big\|
\int_{s}^{-k\tau+jh}
g(r) \dd W_r 
\Big \|
_{L^{4}(\Omega ;\mathbb{R}^{d})} 
\, \dd \zeta \\
& \leq
Ch^{\tfrac{1}{2}}
\int_{0}^{1}
\Big\|
\big(1+
 \| 
 \zeta X_{-k\tau+jh}
 ^{-k\tau}+(\zeta -1) X_{s}^{-k\tau}
 \|
 +\| X_{s}^{-k\tau} \|
\big)^{\gamma -2} \\
& \qquad
\times
\| 
\zeta (X_{-k\tau+jh} 
^{-k\tau}-X_{s}^{-k\tau}) 
\| 
\Big\|
_{L^{4}(\Omega ;\mathbb{R})}
\, \dd \zeta \\
& \leq
Ch
\Big(
1+\sup_{k\in \mathbb{N}} \sup_{t\geq -k\tau} 
{\color{black}
\big(
\| 
X_{t}^{-k\tau}
\|
^{\gamma}
_{L^{4\gamma}(\Omega ;\mathbb{R}^{d})}
\vee
\| 
X_{t}^{-k\tau}
\|
^{2\gamma-2}
_{L^{8\gamma-8}(\Omega ;\mathbb{R}^{d})}
\big)
}
\Big).
\end{split}
\end{equation}
Using \eqref{eq:the_esti_f(t,x)-f(s,x)}, we can easily get
\begin{equation}
\label{eq:f(t,x)-f(s,y)}
\begin{split}
&
\|
f(s,X_{-k\tau+jh}^{-k\tau})
-f(jh,X_{-k\tau+jh}^{-k\tau})
\|
_{L^{2}(\Omega ;\mathbb{R}^{d})} 
\\
& \leq
Ch
\Big(
1+\sup_{k\in \mathbb{N}} \sup_{t\geq -k\tau} 
\| 
X_{t}^{-k\tau}
\|^{\gamma}_{L^{2\gamma}(\Omega ;\mathbb{R}^{d})} 
\Big).
\end{split}
\end{equation}
Inserting \eqref{eq:Phi_1}-\eqref{eq:f(t,x)-f(s,y)} into
\eqref{eq:I_5} implies 
\begin{equation}
\mathbb{I}_5
\leq
C h^2
\Big(
{\color{black}
\big(
\| 
X_{t}^{-k\tau}
\|
^{\gamma}
_{L^{4\gamma}(\Omega ;\mathbb{R}^{d})}
\vee
\| 
X_{t}^{-k\tau}
\|
^{2\gamma-2}
_{L^{8\gamma-8}(\Omega ;\mathbb{R}^{d})}
\big)
}
\Big). \\
\end{equation}
Therefore, from \eqref{eq:R_j|F} it immediately follows that
\begin{equation}
\|
\E
[\mathcal{R}_{j} | \mathcal{F}_{(j-1)h-k\tau}]
\|
_{L^{2}(\Omega ;\mathbb{R}^{d})}
\leq 
C h^{2}
\Big(
{\color{black}
\big(
\| 
X_{t}^{-k\tau}
\|
^{\gamma}
_{L^{4\gamma}(\Omega ;\mathbb{R}^{d})}
\vee
\| 
X_{t}^{-k\tau}
\|
^{2\gamma-2}
_{L^{8\gamma-8}(\Omega ;\mathbb{R}^{d})}
\big)
}
\Big) .
\end{equation}
Considering \eqref{eq:D_j}, define $\hat{v}:=\frac{v}{1+2vh}$, we  get
\begin{equation}
\begin{split}
\E[\| D_j\|^{2}]	
& \leq
(1-\hat{v}h)\E[\| D_{j-1}\|^{2}]+\tfrac{1}{1+2vh}Ch^{3}\\
& \leq
(1-\hat{v}h)^{j}\E[\| D_{0}\|^{2}]
+
\sum_{i=0}^{j-1}
(1-\hat{v}h)^{i}
Ch^{3}\\
&=
(1-\hat{v}h)^{j}
\E[\| D_{0} \|^{2}]
+\frac{1-(1-\hat{v}h)^{j}}{\hat{v}h}Ch^{3}
\end{split}
\end{equation}
By observing $D_{0}=0$,
\begin{equation}
\E[\| D_j\|^{2}]	
\leq
Ch^{2}.
\end{equation}
Then the assertion follows.
\end{proof}

\begin{cor}\label{cor:Ch}
Under Assumptions \ref{ass} ,
let $X^{*}_{t}$ be the random periodic solution of SDE \eqref{eq:Problem_SDE} and $\tilde{X}^{*}_{t}$ be the random periodic solution of the backward Euler numerical approximation.
Then there exists a constant $C$ that depends on $q,\Lambda,f,g$ and $d$ such that 
\begin{equation}
\sup_{t \in \mathcal{T}^{h}} 
\E([\| X^{*}_{t}-\tilde{X}^{*}_{t}\|^{2}])^{1/2}
\leq Ch,
\end{equation}
and $\tilde{X}^{*}_{t}$ satisfies the random periodicity property.
\end{cor}

\begin{proof}[Proof of Corollary \ref{cor:Ch}]
Due to
\begin{equation}
\| X^{*}_{t}-\tilde{X}^{*}_{t}\|^{2}
\leq
\limsup_{k}
\big[\| X^{*}_{t}-X^{-k\tau}_{t}\|^{2}
+\| X^{-k\tau}_{t}-\tilde{X}^{-k\tau}_{t}\|^{2}
+\|\tilde{X}^{-k\tau}_{t}-\tilde{X}^{*}_{t}\|^{2}\big],
\end{equation}
thus the conclusion can be obtained by Theorem \ref{thm:unique_random_periodic_solution}, Theorem \ref{thm:BEM random period solution} and Theorem \ref{thm:error analysis}.

\end{proof}
{\color{black} Corollary \ref{cor:Ch} implies that the order-one convergence can be achieved.}
\section{Numerical experiments}
\label{sec:numerical results}
{\color{black} In this Section, we present two} numerical experiments to {\color{black}support} the theoretical findings. {\color{black} To evaluate the effectiveness of the proposed numerical method under different scenarios, we vary the drift coefficients while maintaining a constant diffusion coefficient, ie, $g(t)=1$. }
\subsection{Example 1}\label{sec:eg1}
{\color{black}In the first example, we test the performance of the backward Euler method \eqref{eq:backward_Euler_method} on a linear drift coefficient as follows:}
\begin{equation}\label{eq:example1}
\dd X^{t_{0}}_{t}=-\pi X^{t_{0}}_{t}\,\dd t+\sin(2\pi t)
\, \dd t+\dd W_{t}.
\end{equation}
{\color{black}To fit into the general form of SDE in \eqref{eq:Problem_SDE}, we have that in this case}
\begin{equation}
\begin{split}
\Lambda
=\pi,\text{ and }
f(t,X^{t_{0}}_{t})
=\sin(2\pi t).
\end{split}
\end{equation}
It is easy to verify that the associated period is 1 and Assumptions \ref{ass} are fulfilled with $\lambda_{1}=\pi$, $c_{f}=1$ and $c_{g}=1$. 
According to Theorem \ref{thm:unique_random_periodic_solution},  {\color{black}SDE \eqref{eq:example1} admits a unique} random periodic solution.
By Theorem \ref{thm:BEM random period solution}, its backward Euler simulation also has a random periodic path.

{\color{black} We will first verify the periodicity by examining the dynamics of two processes under the same realisation $\omega$:  $\tilde{X}^{-5}_{t}(\omega,0.3)$ over
 $-5\leq t\leq 13$ and $\tilde{X}^{-5}_{t}(\theta_{-1}\omega,0.3)$ over $-5\leq t\leq 14$, where $0.3$ is the initial condition of both processes. Because of Theorem \ref{thm:BEM random period solution}, it is expected that $\tilde{X}^{-5}_{t}(\omega,0.3)\approx \tilde{X}^{*}_{t}(\omega)$ and $\tilde{X}^{-5}_{t}(\theta_{-1}\omega,0.3)\approx\tilde{X}^{*}_{t}(\theta_{-1}\omega)$ after a sufficiently long time, and we may then observe that $\tilde{X}^{-5}_{t-1}(\omega,0.3)\approx\tilde{X}^{-5}_{t}(\theta_{-1}\omega,0.3)$ due to the fact $\tilde{X}^{*}_{t-1}(\omega)=\tilde{X}^{*}_{t}(\theta_{-1}\omega)$ in Definition \ref{def:rps}. 
 
{\color{black}
 Figure \ref{fig1} illustrates the evolution of processes before $t=0$, where two trajectories separate from the initial time point and resemble each other with an increasing time gap. For example, it can be observed that $\tilde{X}^{-5}_{-4.35}(\omega,0.3)\approx\tilde{X}^{-5}_{-4.15}(\theta_{-1}\omega,0.3)$, $\tilde{X}^{-5}_{-2.5}(\omega,0.3)\approx\tilde{X}^{-5}_{-1.9}(\theta_{-1}\omega,0.3)$ and $\tilde{X}^{-5}_{-1}(\omega,0.3)\approx\tilde{X}^{-5}_{0}(\theta_{-1}\omega,0.3)$.  Figure \ref{fig2} demonstrates both processes resemble each other with a stable time gap $1$, that is, $\tilde{X}^{-5}_{t-1}(\omega,0.3)\approx\tilde{X}^{-5}_{t}(\theta_{-1}\omega,0.3)$ over $11\leq t\leq 14$.
 }
\begin{figure}
	\centering
	\includegraphics[width=1\linewidth, height=0.4\textheight]
   {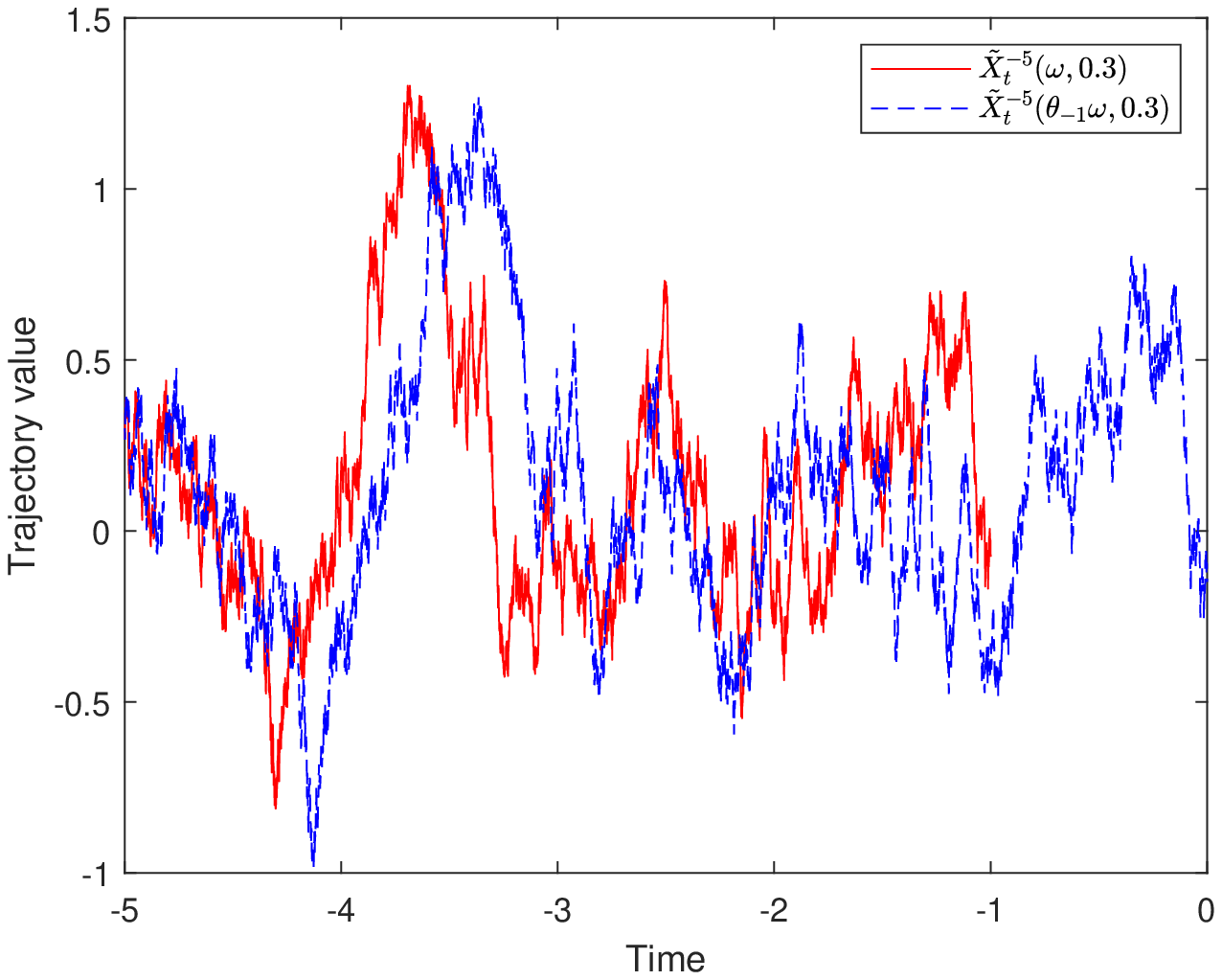}

	\caption[Figure 1.]
	{Simulations of the processes \{$\tilde{X}^{-5}_{t}(\omega,0.3),-5\leq t\leq -1$\} and \{$\tilde{X}^{-5}_{t}(\theta_{-1}\omega,0.3),-5\leq t\leq 0$\}.
 }\label{fig1}
\end{figure}

\begin{figure}
	\centering
	\includegraphics[width=1\linewidth, height=0.4\textheight]{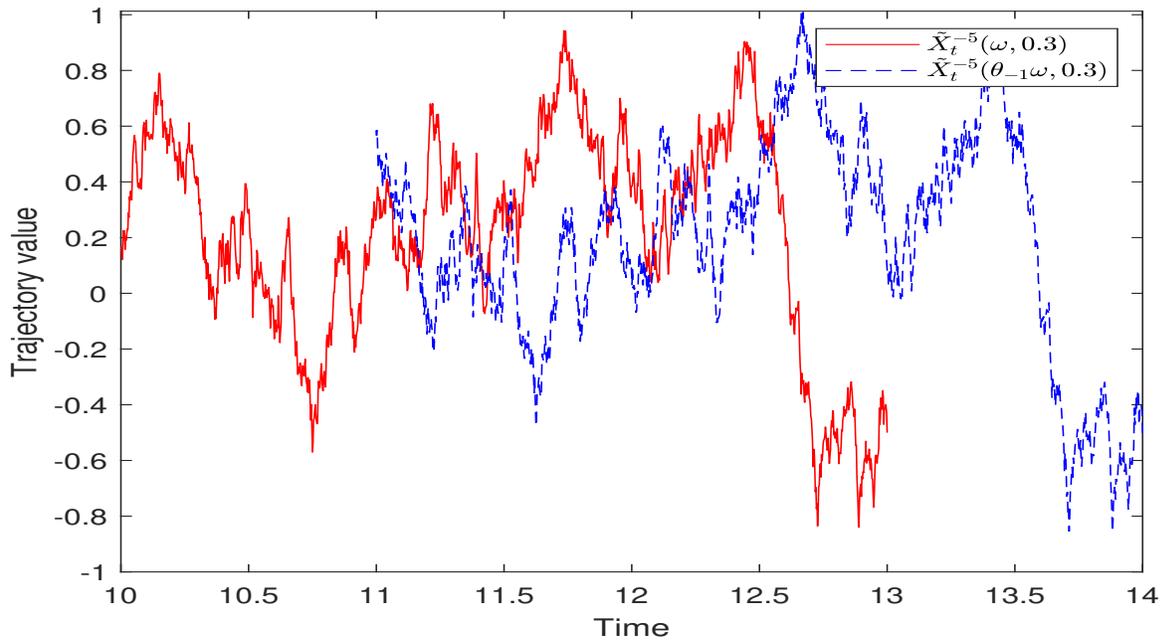}

	\caption[Figure 2.]
	{Simulations of the processes \{$\tilde{X}^{-5}_{t}(\omega,0.3),10\leq t\leq 13$\} and \{$\tilde{X}^{-5}_{t}(\theta_{-1}\omega,0.3),11\leq t\leq 14$\}.
	}\label{fig2}
\end{figure}

{\color{black}
Next, to test the strong convergence rate of backward Euler method, we simulate the solution of \eqref{eq:example1} over $t\in [-5,15]$ with $5000$ samples. The reference solution is obtained via the same numerical method at a fine stepsize $h_{exact}=2^{-15} \times 20$.
We plot in Figure 3 mean-square approximation errors 
$D_{h}$ against six different stepsizes $h=2^{-i} \times 20$, $i=7,8,...,12$ on a log-log scale.}

From Figure \ref{fig3}, one can clearly observe that {\color{black}the mean-square error decreases at a slope close to 1,} consistent with {Theorem \ref{thm:error analysis}}.
{\color{black}
Suppose that the approximation error $D_{h}$ obeys a 
power law relation $D_{h}=Ch^{\kappa}$ for 
$C, \kappa >0$, so that $\log D_{h}=\log C+\kappa \log h$.
Then we do a least squares power law fit for 
$\kappa $ and get the value
0.9836 for the rate $\kappa $  with residual of 0.0697.
Again, this confirms the expected convergence rate in {Theorem \ref{thm:error analysis}.} 
}
\begin{figure}
	\centering	\includegraphics[width=0.8\linewidth, height=0.4\textheight]
 {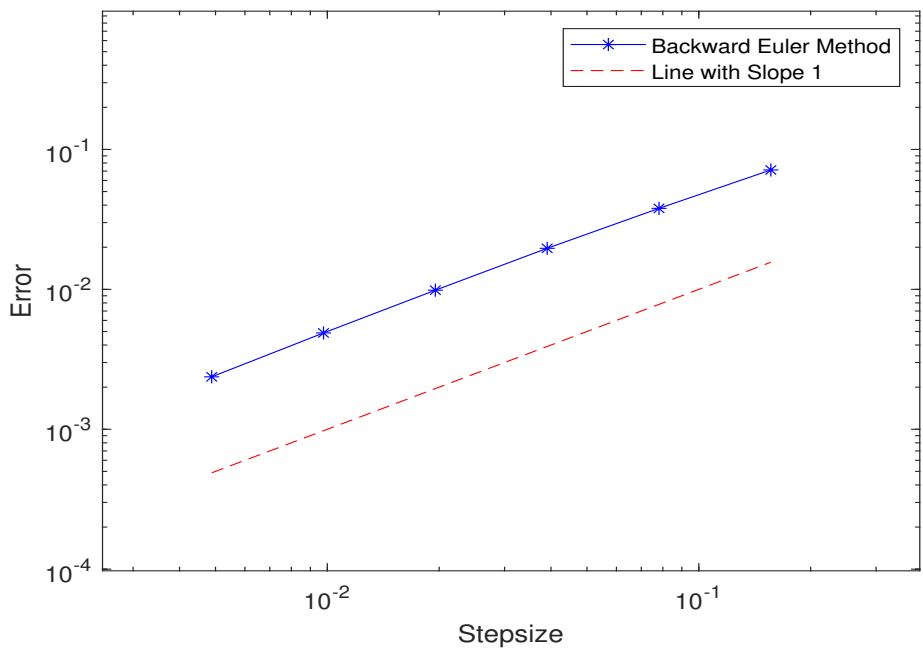}
	\caption[Figure 3.]
	{{\color{black} The mean-square error plot of the backward Euler method \eqref{eq:backward_Euler_method} for simulating the solution of \eqref{eq:example1}.}
	}\label{fig3}
\end{figure}
\subsection{Example 2}
{\color{black}In the second example, we test the performance of the backward Euler method \eqref{eq:backward_Euler_method} on a cubic drift coefficient as follows:}

\begin{equation}
\label{eq:example2}
\dd X^{t_{0}}_{t}=-2\pi X^{t_{0}}_{t}\, \dd t+(X^{t_{0}}_{t}-(X^{t_{0}}_{t})^{3}+\cos(\pi t))
\,\dd t+\dd W_{t}.
\end{equation}
{\color{black}Similarly, we run a similar experiment to verify the periodicity as described in Section \ref{sec:eg1} and present  the identical patterns of $\tilde{X}^{-6}_{t-2}(\omega,0.5)$ over
 $10\leq t\leq 14$ and $\tilde{X}^{-6}_{t}(\theta_{-2}\omega,0.5)$ over $12\leq t\leq 16$ under the same realisation $\omega$ in Figure \ref{fig4}.}

\begin{figure}
	\centering
	\includegraphics[width=1\linewidth, height=0.4\textheight]
    {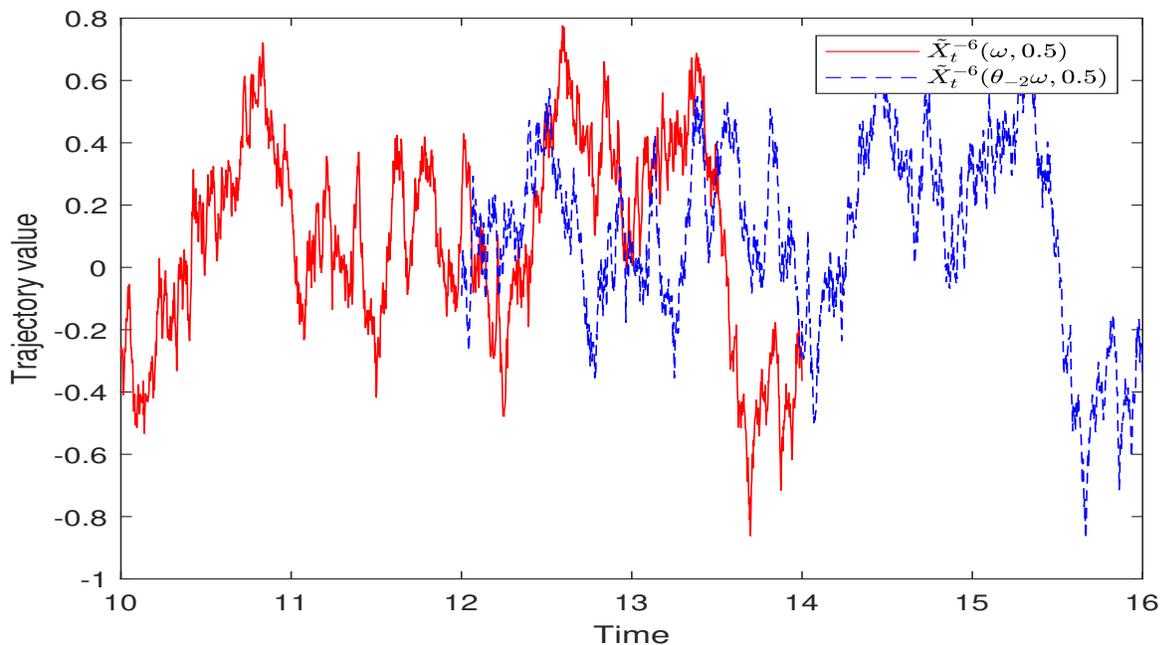}
	
	\caption[Figure 4.]
	{Simulations of the processes \{$\tilde{X}^{-6}_{t}(\omega,0.5),10\leq t\leq 14$\} and \{$\tilde{X}^{-6}_{t}(\theta_{-2}\omega,0.5),12\leq t\leq 16$\}.
	}\label{fig4}
\end{figure}

{\color{black} We also test the performance of the backward Euler method, in terms of mean-square error, in simulating SDE \eqref{eq:example2} over $[-6, 14]$. Figure \ref{figlast} confirms an order-one convergence, again consistent with Theorem \ref{thm:error analysis}.}
{\color{black} A least squares fit
produces a rate 0.9495 with residual of 0.1153 for \eqref{eq:example2}. Hence, numerical result is consistent with strong order of convergence equal to one,
as already revealed in \ref{thm:error analysis}.}
\begin{figure}
	\centering	\includegraphics[width=0.8\linewidth, height=0.4\textheight]
   {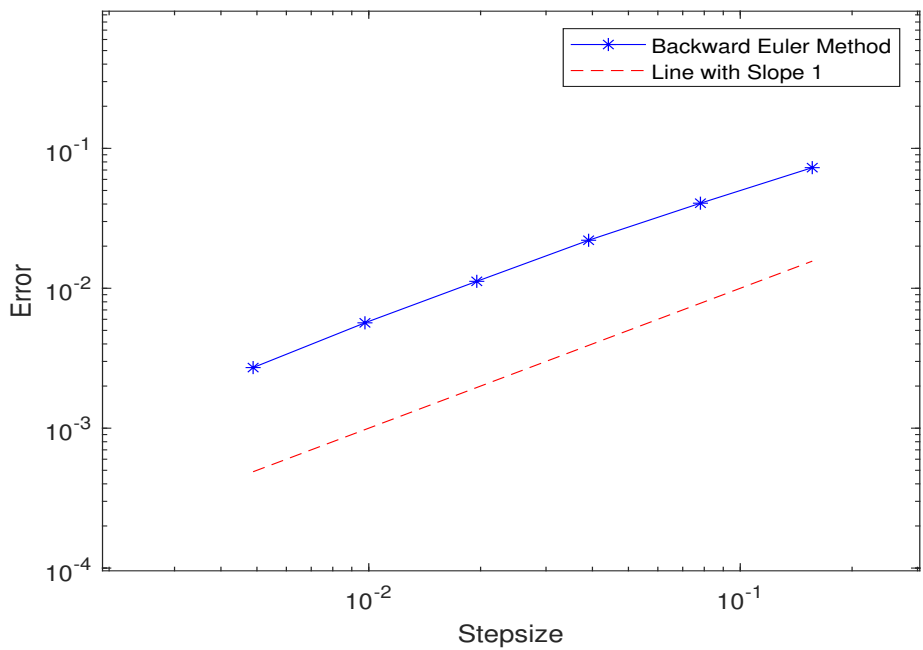}
	\caption[Figure 5.]
	 {
      {\color{black} The mean-square error plot of the backward Euler method \eqref{eq:backward_Euler_method} for simulating the solution of \eqref{eq:example2}.}
	}\label{figlast}
\end{figure}

\end{document}